\documentclass[11pt]{amsart}
\usepackage[dvips]{graphicx}
\usepackage{verbatim, geometry, float, pstricks}
\usepackage[all,2cell]{xy} \UseAllTwocells \SilentMatrices
\newtheorem{theorem}{Theorem}[section]
\newtheorem{lemma}[theorem]{Lemma}
\newtheorem{prop}[theorem]{Proposition}

\newtheorem{cor}[theorem]{Corollary}
\theoremstyle{definition}

\newtheorem{example}[theorem]{Example}

\numberwithin{equation}{section}
\numberwithin{figure}{section}
\numberwithin{table}{section}

\newcommand{\norm}[1]{||#1||}

\def\NN{\mathbb{N}}

\def\RR{\mathbb{R}}

\newcommand{\ran}{{\rm Ran}}

\newcommand{\BBB}[4]{D_{#1}^{#3}D_{#2}^{#4}}
\renewcommand{\emph}[1]{{\bf #1}}
\newcommand{\m}{{\bf m}}
\newcommand{\n}{{\bf n}}
\newcommand{\PP}{{\bf P}}
\newcommand{\TT}{{\bf T}}
\newcommand{\QQ}{{\bf Q}}
\newcommand{\BB}{{\bf B}}
\newcommand{\II}{{\bf I}}
\newcommand{\PPhi}{{\bf \Phi}}
\newcommand{\Mu}{M}
\newcommand{\halpha}{\hat\alpha}
\newcommand{\hbeta}{\hat\beta}

\begin{document}


\title{Discrete Polynomial Blending}

\author{Scott N. Kersey}

\curraddr{Department of Mathematical Science, Georgia Southern University, Statesboro, GA 30460-8093}
\email{scott.kersey@gmail.com}
\thanks{}
\date{\today}
\dedicatory{ }

\begin{abstract}
In this paper we study ``discrete polynomial blending,''
a term used to define a certain discretized version of curve blending
whereby one approximates from the ``sum of tensor product polynomial spaces''
over certain grids.
Our strategy is to combine the theory of Boolean Sum methods with dual bases
connected to the Bernstein basis to construct a new quasi-interpolant for discrete blending.
Our blended element has geometric properties similar to that of the 
Bernstein-B\'ezier tensor product surface patch, 
and rates of approximation that are comparable with
those obtained in tensor product polynomial approximation.
\end{abstract}

\maketitle

\section{Introduction}

In this paper, we study the problem of discrete polynomial blending,
a discretization of the problem of curve blending.
In curve blending, one approximates a bivariate function by interpolating to
a network of curves extracted from the graph of the function.
Discrete blending involves a second level of discretization,
whereby the blended curves are interpolated at finite sets of points.
This is illustrated in Fig. \ref{fblend1}.
In the figure, the blended interpolant would interpolate to the 9 curves 
(4 horizontal and 5 vertical)
in the network, while the discrete blended surface interpolates at $67$ grid points.
In our work, the term ``interpolation'' can be taken loosely to mean interpolation
with respect to a given set of functionals, and may not necessarily imply point evaluation.

\begin{figure}[H]
\scalebox{.4}{
\input{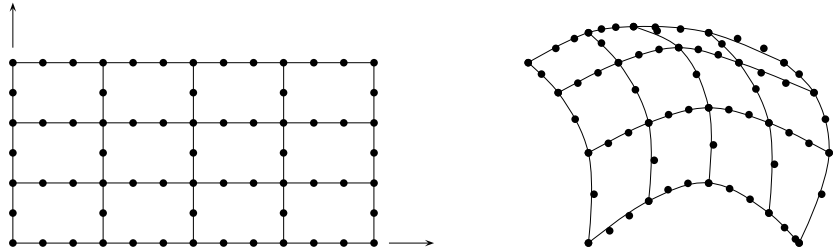}
}
\caption{Discrete Blended Grid and interpolant in the 
polynomial space $\Pi_{12} \otimes \Pi_{3} + \Pi_{5} \otimes \Pi_{7}$
of dimension $13 \times 4 + 5 \times 7 - 5 \times 4= 67$.}
\label{fblend1}
\end{figure}


The grid in Fig. \ref{fblend1} is not uniform due to gaps
between some grid points.
If we filled in these holes we would have a uniform grid with
$13 \times 7 = 91$  points that can be interpolated using tensor product polynomials.
Hence, we can interpolate at just $67$ points rather than $91$,
and in some cases do so with the same (or nearly the same) 
rate of approximation, as we shall show
in this paper.
In discrete polynomial blending, the approximating spaces are not 
generally tensor product polynomial spaces (although that is a special case).
As it turns out, our approximating spaces are the ``sum of tensor product polynomial spaces''.
Hence, our study is one of approximation from the sum of 
tensor product polynomial spaces.

What is known about discrete blending comes mainly from the literature 
on Boolean sum interpolation, sparse grid methods, lower set
interpolation and finite elements.
The topic was perhaps first studied by Biermann \cite{B03}
who constructed polynomial interpolants using the bivariate
Lagrange basis.
The book \cite{DS89} is an excellent summary of Boolean sum methods, 
including an analysis of Biermann interpolation.
In \cite{DP82}, a construction was given that generalized Biermann interpolation 
to interpolation with respect to more general sequences of functionals,
much like we will do in this paper.
Biermann interpolation was generalized to higher dimensions in \cite{D82},
under the title ``d-variate Boolean interpolation'',
and more recently to arbitrary ``lower sets'' in \cite{DF14} 
(which reduces to the results in \cite{D82} for total degree interpolation).

In this paper we construct a new discrete blended quasi-interpolant
based on the Bernstein basis.
To do so, we will bring in some techniques on ``dual basis in subspaces''
that the author has studied in \cite{K13,K14}.
From this we construct dual bases for the space of discrete blending,
and compute approximation estimates.
One of the main contributions is to show our quasi-interpolants achieve
rates of approximation comparable or the same as that of
tensor product interpolation on a larger grid, but with much fewer data points.
This leads us to the construction of a quasi-interpolant
analogous to the serendipity elements in the finite element method.
The results presented in this paper originate from a talk by the author
at the conference on curves and surfaces in Oslo, Norway, in 2012.

The remainder of the paper is organized as follows.
\begin{itemize}
\item The approximating space.
\item Quasi-uniform grids.
\item The Bernstein basis and univariate quasi-interpolant.
\item Discretely blended Bernstein-B\'ezier quasi-interpolants.
\item Order of Approximation.
\item Serendipity Elements.
\item Examples.
\end{itemize}

\section{The Approximating Space}

In discrete polynomial blending, the approximating space
is the algebraic sum of tensor product polynomial spaces.
Let $\m = [m_0, m_1, \ldots, m_r]$
and $\n = [n_0, n_1, \ldots, n_r]$ 
be sequences in $\NN_0^{r+1}$, the space of $(r+1)$-tuples of non-negative integers.
Let $\Pi_{m_k} \otimes \Pi_{n_{r-k}}$ be the tensor product of the spaces
$\Pi_{m_k}$ and $\Pi_{n_{r-k}}$ of polynomials
of degrees at most $m_k$ and $n_{r-k}$, respectively.
Then, we define our approximating spaces as
\begin{equation}
\label{e1}
S_{\m,\n} 
:= \sum_{k=0}^r 
\Pi_{m_k}\otimes \Pi_{n_{r-k}} = 
\Pi_{m_0}\otimes \Pi_{n_{r}} + \cdots + \Pi_{m_r}\otimes\Pi_{n_{0}}.
\end{equation}
We assume that both $\m$ and $\n$ are strictly increasing sequences,
an assumption that we justify by the following lemma.

\begin{lemma}
Let $S_{\m,\n}$ be defined as in (\ref{e1}) for some $\m$ and $\n$ in $\NN_0^{r+1}$.
There exists strictly increasing sequences $\hat\m$ and $\hat\n$ in $\NN_0^{\hat r+1}$
for some $\hat r \leq r$ such that $S_{\hat\m,\hat\n} = S_{\m,\n}$.
\end{lemma}

\begin{proof}
To begin, let $X := S_{\m,\n}$ and let $\hat\m = \m$ and $\hat\n = \n$.
To prove this result, we will rearrange and truncate 
$\hat\m$ and $\hat\n$ until they are strictly increasing. 
Suppose $\hat m_i = \hat m_j$ for some $i$ and $j$.
Then either $\Pi_{\hat m_i} \otimes \Pi_{\hat n_{r-i}} \subset \Pi_{\hat m_{j}} \otimes \Pi_{\hat n_{r-j}}$
or $\Pi_{\hat m_{j}} \otimes \Pi_{\hat n_{r-j}} \subset \Pi_{\hat m_{i}} \otimes \Pi_{\hat\hat  n_{r-i}}$.
In the former case, $\Pi_{\hat m_i} \otimes \Pi_{\hat n_{r-i}}$ can be removed
from the representation in (\ref{e1}) without changing $X$,
and so we remove $\hat m_i$ and $\hat n_{r-i}$ from the sequences 
$\hat \m$ and $\hat \n$.
In the latter case, we remove $\hat m_j$ and $\hat n_{r-j}$.
After removing these unnecessary terms, the terms left in the
revised sequence $\hat \m$ are distinct.
That is, $\hat m_i \neq \hat m_j$ for all $i$ and $j$.
Further, since $S_{\hat\m,\hat\n}$ is not affected by the ordering of the tensor product terms,
we can rearrange the pairs $(\hat m_i,\hat n_{r-i})$ in $\hat \m\times\hat \n$ so that
$\hat\m$ is strictly increasing.
Thus, assume that $\hat\m$ is strictly increasing.
Now, if $\hat \n$ is not strictly increasing, then there exists an index $i$ such that
$\hat m_i < \hat m_{i+1}$ and $\hat n_{r-i} \leq \hat n_{r-(i+1)}$,
in which case
$\Pi_{\hat m_i} \otimes \Pi_{\hat n_{r-i}} \subset \Pi_{\hat m_{i+1}} \otimes \Pi_{\hat n_{r-(i+1)}}$,
and so
$\Pi_{\hat m_i} \otimes \Pi_{\hat n_{r-i}}$ can be removed from the sum without changing $X$.
After trimming away all such terms in the sequence, we are left with $\hat \n$ strictly increasing.
Hence, we are left with strictly increasing sequences $\hat \m$ and $\hat \n$ 
in $\NN_0^{\hat r}$ for some $\hat r$
such that $S_{\hat\m,\hat\n} = X$.
Further, since we are not adding new terms, $\hat r \leq r$.
\end{proof}

A polynomial $p\in S_{\m,\n}$ can be represented $ p = p_1 + \cdots + p_r $ with
$p_k \in \Pi_{m_k} \otimes \Pi_{n_{r-k}}$.
Each $p_k$ can be expressed
$p_k = \sum_{i=0}^{m_k} \sum_{j=0}^{n_{r-k}} \alpha_{ij}^k x^i y^j$
for some coefficients $\alpha_{ij}^k$.
Therefore, the power basis for the space $S_{\m,\n}$ is the union
$$
\bigcup_{k=0}^r \{ x^i y^j : 0 \leq i \leq m_k, \ 0 \leq j \leq n_{r-k} \}.
$$
However, this union is not disjoint.
This basis can visualized by the dots in a \emph{lower grid},
which is defined to be the graph of the \emph{lower set}
$$
L_{\m,\n} := \bigcup_{k=0}^r [0, \ldots, m_k] \times [0, \ldots, n_{r-k}].
$$
The power basis for $S_{\m,\n}$ is therefore
$ \{ x^i y^j : (i,j) \in L_{\m,\n} \}$,
and the dimension of $S_{\m,\n}$  is the number of dots in the lower grid.
By counting the distinct dots in the lower grid, we arrive at the following:

\begin{prop}
\label{prop0}
The space $S_{\m,\n}$ is a vector space of dimension
\begin{align*}
\dim(S_{\m,\n}) &= \sum_{k=0}^r (m_k-m_{k-1}) (n_{r-k} +1),
\end{align*}
with $m_{-1} := -1$.
\end{prop}

In Tbl. \ref{f1a}, the dimension of the approximating space is given for
a few choices of $\m$ and $\n$.
In Fig. \ref{f1b}, the corresponding lower grids are plotted.

\begin{table}[H]
\begin{tabular}{|l|l|l|l|}
\hline
$\m$ & $\n$ & $S_{\m,\n}$ & $\dim(S_{\m,\n})$ \\
\hline
$[2]$ & $[2]$ & $\Pi_2 \otimes \Pi_2$ & $(2+1)(2+1) = 9$ \\
$[1,2]$ & $[1,2]$ & $\Pi_1 \otimes \Pi_2 + \Pi_2 \otimes \Pi_1$ & 
$(1+1)(2+1) + (2-1)(1+1) = 8$ \\
$[1,3,6]$ & $[1,3,4]$ &
$\Pi_1 \otimes \Pi_4 + \Pi_3 \otimes \Pi_3 + \Pi_6 \otimes \Pi_1$ &
$(1+1)(4+1)+(3-1)(3+1)+(6-3)(1+1)  = 24$ \\
\hline
\end{tabular}
\caption{
Approximating spaces for: $\m=\n=[2]$, $\m=\n=[1,2]$, $\m= [1,3,6]$ and $\n = [1,3,4]$.
}
\label{f1a}
\end{table}

\begin{figure}[H]
\scalebox{0.5}{
\begin{pspicture}(7,5)
\psdots[dotscale=1.5](0,0)(0,1)(0,2)(1,0)(1,1)(1,2)(2,0)(2,1)(2,2)
\end{pspicture}}
\scalebox{0.5}{
\begin{pspicture}(7,5)
\psdots[dotscale=1.5](0,0)(0,1)(0,2)(1,0)(1,1)(1,2)(2,0)(2,1)
\end{pspicture}}
\scalebox{0.5}{
\begin{pspicture}(7,5)
\psdots[dotscale=1.5](0,0)(0,1)(0,2)(0,3)(0,4)(1,0)(1,1)(1,2)(1,3)(1,4)(2,0)(2,1)(2,2)(2,3)
(3,0)(3,1)(3,2)(3,3)(4,0)(4,1)(5,0)(5,1)(6,0)(6,1)
\end{pspicture}}
\caption{ Lower grids for $\m=\n=[2]$, $\m=\n=[1,2]$, $\m= [1,3,6]$ and $\n = [1,3,4]$. }
\label{f1b}
\end{figure}

\section{Quasi-Uniform Grids}

Our discretely blended surfaces are defined over certain quasi-uniform grids, defined as follows.
As above, $\m = \{m_0, \ldots, m_r\}$ and $\n = \{n_0, \ldots, n_r\}$ 
are strictly increasing sequences in $\NN_0^r$.
We assume moreover that 
$m_k$ divides $m_{k+1}$ and $n_{k}$ divides $n_{k+1}$ for $k=0, \ldots, r-1$.
For $k=0, \ldots, r$, let
$$\alpha^{k} = [i m_r/m_k: i = 0, \ldots, m_k]$$
and
$$\beta^{k} = [i n_r/n_k: i = 0, \ldots, n_k].$$
Note that $\alpha^{k}$ is a sequence of $m_k+1$ uniformly spaced points
from $0$ to $m_r$,
and
$\beta^{k}$ is a sequence of $n_k+1$ uniformly spaces points
from $0$ to $n_r$.
Since $m_k | m_r$ and $n_k | n_r$, these are integer sequences.
Moreover, because $m_k | m_{k+1}$ and $n_k | n_{k+1}$,
it follows that $\alpha^{k} \subset \alpha^{k+1}$
and $\beta^{k} \subset \beta^{k+1}$.
Then, we define our \emph{quasi-uniform set} as 
$$
G_{\m,\n} 
:= \bigcup_{k=0}^r \Big\{(\alpha_i^{k}, \beta_j^{{r-k}}) 
: i = 0, \ldots, m_k, \ j=0, \ldots, n_{r-k} \Big\}.
$$
We call the graph of the quasi-uniform set $G_{\m,\n}$ a \emph{quasi-uniform grid}.
By the assumptions on $\m$ and $\n$, the the number of dots in 
our quasi-uniform grids match the dimension of the spaces $S_{\m,\n}$.
In fact, the graph of $G_{\m,\n}$ is a permutation of the graph of $L_{\m,\n}$.
To construct the permutation, let
$$
\alpha = [\alpha^0, \alpha^1-\alpha^0, \ldots, \alpha^r-\alpha^{r-1}]
$$
and
$$
\beta = [\beta^0, \beta^1-\beta^0, \ldots, \beta^r-\beta^{r-1}],
$$
with $\alpha^{k+1}-\alpha^k$ and $\beta^{k+1}-\beta^k$ defined as the set difference.
Then, $L_{\m,\n} = G_{\m,\n}(\alpha,\beta)$.

An example is provided in Fig. \ref{f5},
where lower and quasi-uniform grids are plotted for the space
$$
S_{[3,6,12,24],[2,4,8,16]} = \Pi_3 \times \Pi_{16} + 
  \Pi_6 \times \Pi_8 + \Pi_{12} \times \Pi_4 + \Pi_{24} \times \Pi_2.
$$
Here, 
\begin{align*}
\alpha^0 &= [0,8,16,24] \\
\alpha^1 &= [0,4,8,12,16,20,24] \\
\alpha^2 &= [0,2,4,6,8,10,12,14,16,18,20,22,24] \\
\alpha^3 &= [0,1,2,3,4,5,6,7,8,9,10,11,12,13,14,15,16,17,18,19,20,21,22,23,24] \\
\beta^0 &= [0,8,16] \\
\beta^1 &= [0,4,8,12,16] \\
\beta^2 &= [0,2,4,6,8,10,12,14,16] \\
\beta^3 &= [0,1,2,3,4,5,6,7,8,9,10,11,12,13,14,15,16].
\end{align*}
Therefore,
\begin{align*}
\alpha
&= 
[\underbrace{0,8,16,24}_{\alpha^0},
 \underbrace{4,12,20}_{\alpha^1-\alpha_0},
 \underbrace{2,6,10,14,18,22}_{\alpha^2-\alpha_1},
 \underbrace{1,3,5,7,9,11,13,15,17,19,21,23}_{\alpha^3-\alpha^2}] \\
\beta
&= [\underbrace{0,8,16}_{\beta^0},
    \underbrace{4,12}_{\beta^1-\beta^0},
    \underbrace{2,6,10,14}_{\beta^2-\beta^1},
    \underbrace{1,3,5,7,9,11,13,15}_{\beta^3-\beta^2}].
\end{align*}
The dimension of $S_{\m,\n}$ is $161$, which matches the number of grid points.

\begin{figure}[H]
\centering
\scalebox{0.3}{
\begin{pspicture}(24,16)
\psdots[dotscale=2]
(3,16)(6,8)(12,4)(24,2)
(0,0)(0,1)(0,2)(0,3)(0,4)(0,5)(0,6)(0,7)(0,8)(0,9)(0,10)(0,11)(0,12)(0,13)(0,14)(0,15)(0,16)(1,0)(1,1)(1,2)(1,3)(1,4)(1,5)(1,6)(1,7)(1,8)(1,9)(1,10)(1,11)(1,12)(1,13)(1,14)(1,15)(1,16)(2,0)(2,1)(2,2)(2,3)(2,4)(2,5)(2,6)(2,7)(2,8)(2,9)(2,10)(2,11)(2,12)(2,13)(2,14)(2,15)(2,16)(3,0)(3,1)(3,2)(3,3)(3,4)(3,5)(3,6)(3,7)(3,8)(3,9)(3,10)(3,11)(3,12)(3,13)(3,14)(3,15)(0,0)(0,1)(0,2)(0,3)(0,4)(0,5)(0,6)(0,7)(0,8)(1,0)(1,1)(1,2)(1,3)(1,4)(1,5)(1,6)(1,7)(1,8)(2,0)(2,1)(2,2)(2,3)(2,4)(2,5)(2,6)(2,7)(2,8)(3,0)(3,1)(3,2)(3,3)(3,4)(3,5)(3,6)(3,7)(3,8)(4,0)(4,1)(4,2)(4,3)(4,4)(4,5)(4,6)(4,7)(4,8)(5,0)(5,1)(5,2)(5,3)(5,4)(5,5)(5,6)(5,7)(5,8)(6,0)(6,1)(6,2)(6,3)(6,4)(6,5)(6,6)(6,7)(0,0)(0,1)(0,2)(0,3)(0,4)(1,0)(1,1)(1,2)(1,3)(1,4)(2,0)(2,1)(2,2)(2,3)(2,4)(3,0)(3,1)(3,2)(3,3)(3,4)(4,0)(4,1)(4,2)(4,3)(4,4)(5,0)(5,1)(5,2)(5,3)(5,4)(6,0)(6,1)(6,2)(6,3)(6,4)(7,0)(7,1)(7,2)(7,3)(7,4)(8,0)(8,1)(8,2)(8,3)(8,4)(9,0)(9,1)(9,2)(9,3)(9,4)(10,0)(10,1)(10,2)(10,3)(10,4)(11,0)(11,1)(11,2)(11,3)(11,4)(12,0)(12,1)(12,2)(12,3)(0,0)(0,1)(0,2)(1,0)(1,1)(1,2)(2,0)(2,1)(2,2)(3,0)(3,1)(3,2)(4,0)(4,1)(4,2)(5,0)(5,1)(5,2)(6,0)(6,1)(6,2)(7,0)(7,1)(7,2)(8,0)(8,1)(8,2)(9,0)(9,1)(9,2)(10,0)(10,1)(10,2)(11,0)(11,1)(11,2)(12,0)(12,1)(12,2)(13,0)(13,1)(13,2)(14,0)(14,1)(14,2)(15,0)(15,1)(15,2)(16,0)(16,1)(16,2)(17,0)(17,1)(17,2)(18,0)(18,1)(18,2)(19,0)(19,1)(19,2)(20,0)(20,1)(20,2)(21,0)(21,1)(21,2)(22,0)(22,1)(22,2)(23,0)(23,1)(23,2)(24,0)(24,1)
\end{pspicture}}
\qquad
\scalebox{0.3}{
\begin{pspicture}(24,16)
\psdots[dotscale=2]
(24,16)
(0,0)(0,1)(0,2)(0,3)(0,4)(0,5)(0,6)(0,7)(0,8)(0,9)(0,10)(0,11)(0,12)(0,13)(0,14)(0,15)(0,16)(8,0)(8,1)(8,2)(8,3)(8,4)(8,5)(8,6)(8,7)(8,8)(8,9)(8,10)(8,11)(8,12)(8,13)(8,14)(8,15)(8,16)(16,0)(16,1)(16,2)(16,3)(16,4)(16,5)(16,6)(16,7)(16,8)(16,9)(16,10)(16,11)(16,12)(16,13)(16,14)(16,15)(16,16)(24,0)(24,1)(24,3)(24,4)(24,5)(24,6)(24,7)(24,8)(24,9)(24,10)(24,11)(24,12)(24,13)(24,14)(24,15)(0,0)(0,2)(0,4)(0,6)(0,8)(0,10)(0,12)(0,14)(0,16)(4,0)(4,2)(4,4)(4,6)(4,8)(4,10)(4,12)(4,14)(4,16)(8,0)(8,2)(8,4)(8,6)(8,8)(8,10)(8,12)(8,14)(8,16)(12,0)(12,2)(12,4)(12,6)(12,8)(12,10)(12,12)(12,14)(16,0)(16,2)(16,4)(16,6)(16,8)(16,10)(16,12)(16,14)(16,16)(20,0)(20,2)(20,4)(20,6)(20,8)(20,10)(20,12)(20,14)(20,16)(24,0)(24,4)(24,6)(24,8)(24,10)(24,12)(24,14)(0,0)(0,4)(0,8)(0,12)(0,16)(2,0)(2,4)(2,8)(2,12)(2,16)(4,0)(4,4)(4,8)(4,12)(4,16)(6,0)(6,4)(6,12)(6,16)(8,0)(8,4)(8,8)(8,12)(8,16)(10,0)(10,4)(10,8)(10,12)(10,16)(12,0)(12,4)(12,8)(12,12)(12,16)(14,0)(14,4)(14,8)(14,12)(14,16)(16,0)(16,4)(16,8)(16,12)(16,16)(18,0)(18,4)(18,8)(18,12)(18,16)(20,0)(20,4)(20,8)(20,12)(20,16)(22,0)(22,4)(22,8)(22,12)(22,16)(24,0)(24,4)(24,8)(24,12)(0,0)(0,8)(0,16)(1,0)(1,8)(1,16)(2,0)(2,8)(2,16)(3,0)(3,8)(3,16)(4,0)(4,8)(4,16)(5,0)(5,8)(5,16)(6,0)(6,16)(7,0)(7,8)(7,16)(8,0)(8,8)(8,16)(9,0)(9,8)(9,16)(10,0)(10,8)(10,16)(11,0)(11,8)(11,16)(12,0)(12,8)(12,16)(13,0)(13,8)(13,16)(14,0)(14,8)(14,16)(15,0)(15,8)(15,16)(16,0)(16,8)(16,16)(17,0)(17,8)(17,16)(18,0)(18,8)(18,16)(19,0)(19,8)(19,16)(20,0)(20,8)(20,16)(21,0)(21,8)(21,16)(22,0)(22,8)(22,16)(23,0)(23,8)(23,16)(24,0)(24,8)(24,2)(6,8) \end{pspicture}}
\caption{Lower grid $L_{\m,\n}$ and quasi-uniform grid $G_{\m,\n}$ 
for $\m = [3,6,12,24]$ and $\n = [2,4,8,16]$.}
\label{f5}
\end{figure}

\section{The Bernstein Basis and Quasi-Interpolation Projectors}

In this section we construct univariate projectors based on quasi-interpolation.
Our construction begins with the Bernstein basis for univariate polynomials.
We assume as before that
$\m$ and $\n$ are strictly increasing sequences in $\NN_0^r$ 
such that $m_k | m_{k+1}$ and $n_k | n_{k+1}$ for $k=0, \ldots, r-1$.  
Let $B^{m_k}= [B_0^{m_k}, \ldots, B_m^{m_k}]$ 
be the Bernstein basis for $\Pi_{{m_k}}$,
scaled to the interval $[a,b]$,
and let 
$B^{n_k}= [B_0^{n_k}, \ldots, B_n^{n_k}]$ be the Bernstein basis for $\Pi_{{n_k}}$ 
scaled to the interval $[c,d]$.
Hence,
\begin{align*}
B_i^{m_k}&=\binom{m_k}{i} \Big(\frac{b-\cdot}{b-a}\Big)^{m_k-i}\Big(\frac{\cdot-a}{b-a}\Big)^i
\quad \text{and} \quad
B_j^{n_k}=\binom{n_k}{j} \Big(\frac{d-\cdot}{d-c}\Big)^{n_k-j}\Big(\frac{\cdot-c}{d-c}\Big)^j.
\end{align*}
We view these bases as row vectors.
Hence, for coefficient sequences
$\alpha \in \RR^{m_k+1}$,
viewed as column vectors,
we have the compact representation
$$
p = B^{m_k} \alpha = \sum_{i=0}^{m_k} \alpha_i B_i^{m_k}
$$
for polynomials $p \in\Pi_{m_k}$.

Next, we construct functionals dual to $B^{m_k}$ and $B^{n_k}$ defined over
continuous functions.
To do so, we follow the dual functionals constructed in Section 2.16 in \cite{SL2007} for the
multivariate Bernstein basis.
For this, we let
$\Delta_{m_k} = [\delta_{x_0^{m_k}}, \ldots, \delta_{x_{m_k}^{m_k}}]$
be the map 
$$ \Delta_{m_k} : f \mapsto [f(x_0^{m_k}), \ldots, f(x_{m_k}^{m_k})] $$
at points $x^{m_k}_i = a+\frac{i}{{m_k}} (b-a)$.
Likewise, we define $\Delta_{n_k}$ at $y^{n_k}_i = c+\frac{i}{{n_k}} (d-c)$.
Let $T_{m_k}$ be the $({m_k}+1)\times({m_k}+1)$ matrix
$$T_{m_k} := \Delta_{{m_k}}^TB^{m_k} = [\delta_{x_i^{m_k}} B_j^{m_k}: 0 \leq i,j \leq {m_k} ],$$ 
and we define $\Lambda_{m_k} := \Delta_{m_k} T_{m_k}^{-T}$.
Thus, $\Lambda_{m_k}  = [\lambda_0^{m_k}, \ldots, \lambda_{m_k}^{m_k}]$ is a vector-map of
${m_k}+1$ functionals defined on continuous functions.
Likewise, we construct $\Lambda_{n_k}$ similarly.
Duality is verified next.

\begin{lemma}
$\Lambda_{m_k}$ is dual to $B^{m_k}$ and $\Lambda_{n_k}$ is dual to $B^{n_k}$.
\end{lemma}

\begin{proof}
Duality of $\Lambda_{m_k}$ follows by
$$\Lambda_{m_k}^TB^{m_k}
 = \big(\Delta_{m_k} T_{m_k}^{-T}\big)^T B^{m_k} 
 = T_{m_k}^{-1} \Delta_{m_k}^T  B^{m_k} 
 = T_{m_k}^{-1} T_{m_k} = I,
$$
where $I$ is the $({m_k}+1)\times({m_k}+1)$ identity matrix.
Hence, $\lambda_i^{{m_k}} B_j^{{m_k}} = \delta_{ij}$.
Duality of $\Lambda_{n_k}$ is proved similarly.
\end{proof}

Following the construction laid out in \cite{DP82},
we construct bases in $\Pi_{m_k}$ that are dual to subsets of the functionals in
$\Lambda^{m_r}$.
The subsets are 
$$\Mu_{m_k} = [\mu^{m_k}_0, \ldots, \mu_{m_k}^{m_k}]
:=[\lambda_{\alpha^k_0}^{m_r}, \ldots, \lambda_{\alpha^k_{m_k}}^{m_r}]$$
and
$$\Mu_{n_k} = [\mu^{n_k}_0, \ldots, \mu_{n_k}^{n_k}]
:=[\lambda_{\beta^k_0}^{n_r}, \ldots, \lambda_{\beta^k_{n_k}}^{n_r}],$$
with respect to the grid-points $(\alpha_i^k, \beta_j^k)$ in $G_{\m,\n}$.
Equivalently, we write
$\Mu_{m_k}  = \Lambda_{m_r}(\alpha^k)$
and
$\Mu_{n_k}  = \Lambda_{n_r}(\beta^k)$.
Based on results by the author (\cite{K13,K14}),
$\Mu_{m_k}$ is linearly independent on $\Pi_{m_k}$ and
$\Mu_{n_k}$ is linearly independent on $\Pi_{n_k}$.
Hence, we can construct dual bases.
Let
$D^{m_k} = [D_0^{m_k}, \ldots, D^{m_k}_{m_k}]$ be the basis for
$\Pi_{m_k}$ dual to $\Mu_{m_k}$,
and let
$D^{n_k} = [D_0^{n_k}, \ldots, D^{n_k}_{n_k}]$ be the basis for
$\Pi_{n_k}$ dual to $\Mu_{n_k}$.

We can find explicit representations for the bases $D^{m_k}$
and $D^{n_k}$ as follows.
Since $\Pi_{m_k}$ embeds into $\Pi_{m_r}$, there is a 
matrix $E_{m_k}^{m_r}$  (the degree elevation matrix) such that
$B^{m_k} = B^{m_r} E_{m_k}^{m_k}$.
Recall that $\Lambda_{m_r}$ is dual to $B^{m_r}$.
Therefore, $\Lambda_{m_r}^T B^{m_r} = I$, and so
the matrix $E_{m_k}^{m_r}$
can be computed from
$$
E_{m_k}^{m_r} = \Lambda_{m_r}^T B^{m_r} E_{m_k}^{m_r} = \Lambda_{m_r}^T B^{m_k}.
$$
Since $D^{m_k}$ is a basis for $\Pi_{m_k}$, we can find a transformation
$A$ so that $D^{m_k} = B^{m_k} A$.
By duality,
$$
I = \Mu_{m_k}^{T}  D^{m_k}
= \Lambda_{m_r}(\alpha^k)^T  B^{m_k} A
= I(\alpha^k,:) \Lambda_{m_r}^T B^{m_k} A
= I(\alpha^k,:) E_{m_k}^{n_k} A
= E_{m_k}^{n_k}(\alpha^k,:) A,
$$
and so $A = E_{m_k}^{m_r} (\alpha^k,:)^{-1}$.
Note that invertibility of $E_{m_k}^{m_r}(\alpha^k,:)$ and $E_{n_k}^{n_r}(\beta^k,:)$ 
follows from linear independence of $M_{m_k}$ and $M_{n_k}$, 
as was established in (\cite{K13,K14}).
Therefore, 
$$
D^{m_k} = B^{m_k} E_{m_k}^{m_r} (\alpha^k,:)^{-1}
$$
with
$E_{m_k}^{m_r} = \Lambda_{m_r}^T B^{m_k}$.
Likewise,
$$
D^{n_k} = B^{n_k} E_{n_k}^{n_r} (\beta^k,:)^{-1}
$$
with
$E_{n_k}^{n_r} = \Lambda_{n_r}^T B^{n_k}$.

We define our univariate quasi-interpolants as follows:
$$
P_{m_k} : f \mapsto  D^{m_k} \Mu_{m_k}^{T} f
= \sum_{i=0}^{m_k} \big(\mu_{i}^{m_k}f\big)\  D_i^{m_k}
$$
and
$$
Q_{n_k}  : g \mapsto D^{n_k} \Mu_{n_k}^T f
= \sum_{j=0}^{n_k} \big(\mu_{j}^{n_k}f\big)\  D_j^{n_k}.
$$
Now, we can verify a couple facts.

\begin{prop}
\
\begin{enumerate}
\item $P_{m_k}$ and $Q_{n_k}$ are linear projectors.
\item $P_{m_k}P_{m_\ell} = P_{m_\ell} = P_{m_\ell} P_{m_k}$ 
and $Q_{n_k}Q_{n_\ell} = Q_{n_\ell} = Q_{n_\ell} Q_{n_k}$ when $\ell \leq k$.
\end{enumerate}
\end{prop}

\begin{proof}
For (1), linearity is straight forward, and idempotency follows by duality:
$$
P_{m_k}^2f = D^{m_k} \Mu_{m_k}^T \Big(D^{m_k} \Mu_{m_k}^Tf\Big)
= D^{m_k} \Big(\Mu_{m_k}^T D^{m_k} \Big)\Mu_{m_k}^Tf
= D^{m_k} \Mu_{m_k}^Tf
= P_{m_k} f,
$$
For (2), define, as above, the degree elevation matrix
$E_{m_\ell}^{m_k}$ embeds the basis for $\Pi_{m_\ell}$ into $\Pi_{m_k}$ by
$B^{m_\ell} = B^{m_k} E_{m_\ell}^{m_k}$.
Thus,
$$
P_{m_k}P_{m_\ell} 
  = B^{m_k} \Mu_{m_k}^T \big(B^{m_\ell} \Mu_{m_\ell}^T\big)
  = B^{m_k} \big(\Mu_{m_k}^T B^{m_k} \big)E_{m_\ell}^{m_k} \Mu_{m_\ell}^T
  = B^{m_k} E_{m_\ell}^{m_k} \Mu_{m_\ell}^T
  = B^{m_\ell} \Mu_{m_\ell}^T
  = P_{m_\ell},
$$
and
$$
P_{m_\ell} P_{m_k}
  = B^{m_\ell} \Mu_{m_\ell}^T \big(B^{m_k} \Mu_{m_k}^T\big)
  = B^{m_\ell} \Mu_{m_\ell}^T \big(B^{m_k} E_{m_\ell}^{m_k}\big)\Mu_{m_\ell}^T
  = B^{m_\ell} \big(\Mu_{m_\ell}^T B^{n_\ell}\big) \Mu_{m_\ell}^T
  = B^{m_\ell} \Mu_{m_\ell}^T
  = P_{m_\ell.}
$$
The proofs for $Q_{n_k}$ are analogous.
\end{proof}

Now, we establish bounds for these projectors.

\begin{theorem}
\label{thm3}
For all $f \in C[a,b]$ and $g\in C[c,d]$,
$$
\norm{P_{m_k} f}_{\infty, [a,b]}  \leq C_{m_r,m_k} \norm{f}_{\infty,[a,b]}
$$
and
$$
\norm{Q_{n_k} g}_{\infty, [c,d]}  \leq C_{n_r,n_k} \norm{g}_{\infty,[c,d]},
$$
where $C_{m_r,m_k}$ is a constant depending only on $m_r$ and $m_k$,
and $C_{n_r,n_k}$ is a constant depending only on $n_r$ and $n_k$.
\end{theorem}

\begin{proof}
The proofs of the two inequalities are identical, hence we'll prove just the first.
From above,
$$
P_{m_k}f(x) = \sum_{i=0}^{m_k} (\mu_{i}^{m_k}f)  D_i^{m_k}(x)
$$
with
$D^{m_k} = E_{m_k}^{m_r}(\alpha_k,:)^{-1} B^{m_k}$.
By Lemma 2.4.1 in \cite{SL2007},
$|\mu_i^{m_k}f| = |\lambda_{\alpha^k_i}^{n_r} f| \leq C_{m_r} \norm{f}_{\infty,[a,b]}$
for some constant $C_{m_r}$ depending only on $m_r$.
Thus,
\begin{align*}
|P_{m_k}f(x)|  &\leq \sum_{i=0}^{m_k} |\mu_{i}^{m_k}f|  |D_i^{m_k}(x)| 
 \leq C_{m_r} \norm{f}_{\infty,[a,b]} \sum_{i=0}^{m_k} |D_i^{m_k}(x)|.
\end{align*}
Let $A := E_{m_k}^{m_r} (\alpha^k,:)^{-1}$.
Then, for $x \in [a,b]$,
\begin{align*}
 \sum_{i=0}^{m_k} |D_i^{m_k}(x)| 
 &= \sum_{i=0}^{m_k} |A(i,:) B^{m_k}(x)| 
 = \sum_{i=0}^{m_k} \Big|\sum_{j=0}^{m_k} A(i,j) B_j^{m_k}(x)\Big| \\
 &\leq \norm{A}_\infty \sum_{i=0}^{m_k}   \sum_{j=0}^{m_k} B_j^{m_k}(x) 
 = \norm{A}_\infty \sum_{i=0}^{m_k}   1 
 = (m_k+1) \norm{A}_\infty,
\end{align*}
with
$$
\norm{A}_\infty = \norm{E_{m_k}^{m_r} (\alpha^k,:)^{-1}}_\infty
 = \frac{1}{\min_{\norm{x}_\infty = 1} \big(\norm{E_{m_k}^{m_r} (\alpha^k,:)x}_\infty\big)^{-1}}.
$$
Thus, we have established the desired result with
$$
C_{m_r,m_k} :=  C_{m_r} (m_k+1)\norm{E_{m_k}^{m_r} (\alpha^k,:)^{-1}}_\infty.
$$
\end{proof}

\section{Discretely Blended Quasi-Interpolants}

In this section we construct quasi-interpolants over our 
quasi-uniform grids $G_{\m,\n}$.
Our construction is based on Generalized Biermann Interpolation (\cite{DP82}) 
and Boolean Sum methods (\cite{DS89}), combined with the author's work on dual bases 
in subspaces based on the Bernstein basis (\cite{K13,K14}).
We begin by extending the univariate projectors
to bivariate projectors.
Let $F$ be a continuous bivariate function.
Then, with $I$ the identity operator,
let $\PP_{m_k} = P_{m_k} \times I$,
a projector that acts only on the first coordinate of bivariate functions,
and let $\QQ_{n_k} = I \times Q_{n_k}$,
a projector that acts only on the second coordinate.
Therefore, we can define tensor product projectors as
$$
\PP_{m_k}\QQ_{n_{r-k}} F(u,v)  = 
\sum_{i=0}^{m_k} \sum_{j=0}^{n_{r-k}} 
(\mu_i^{m_k} \times \mu_j^{n_{r-k}}) F
\ D_i^{m_k}(u) D_j^{n_{r-k}}(v).
$$
Like Biermann interpolation in \cite{DP82, DS89}, our discretely blended projector
is defined as a Boolean sum of tensor product projectors as follows:
\begin{align*}
\BB_{\m,\n} := & \bigoplus_{k=0}^r \PP_{m_k}\QQ_{n_{r-k}}
= \PP_{m_0} \oplus \QQ_{n_r}  + 
\PP_{m_1} \oplus \QQ_{n_{r-1}}  +  \cdots + 
\PP_{m_r} \oplus \QQ_{n_0},
\end{align*}
with Boolean sum
$$
\PP_{m_k} \oplus \QQ_{n_{\ell}} := \PP_{m_k} + \QQ_{n_{\ell}} - \PP_{m_k} \QQ_{n_{\ell}}.
$$
From (\cite{DP82,DS89}) we have:
\begin{prop}
\begin{enumerate}
\ \\
\item $\BB_{\m,\n}$ is a linear projector onto $S_{\m,\n}$.
\item $\BB_{\m,\n} = \sum_{k=0}^r \PP_{m_k}\QQ_{n_{r-k}} 
  - \sum_{k=0}^{r-1} \PP_{m_k}\QQ_{n_{r-k-1}}. $
\end{enumerate}
\end{prop}

The second part of this proposition is important as reduces the number
of terms needed to represent $\BB_{\m,\n}$.
In view of Theorem \ref{thm3}, we have the following estimates:

\begin{theorem}
\label{thm4}
For $F \in C([a,b]\times[c,d])$: 
\begin{enumerate}
\item
$\displaystyle
\norm{\PP_{m_k} \QQ_{n_\ell} F}_{\infty, [a,b]\times[c,d]}  
\leq C_{m_r,m_k} C_{n_r,n_\ell} \norm{F}_{\infty,[a,b]\times[c,d]},$
\item
$\displaystyle
\norm{\PP_{m_k} \oplus \QQ_{n_\ell} F}_{\infty, [a,b]\times[c,d]}  
\leq (C_{m_r,m_k} + C_{n_r,n_\ell} + C_{m_r,m_k} C_{n_r,n_\ell})
\norm{F}_{\infty,[a,b]\times[c,d]},$
\item
$\displaystyle
\norm{\BB_{\m,\n}F}_{\infty, [a,b]\times[c,d]}  
\leq 
\Big(\sum_{k=0}^r C_{m_r,m_k} C_{n_r,n_{n-k}} 
+ \sum_{k=0}^{r-1} C_{m_r,m_k} C_{n_r,n_{n-k-1}}  \Big)
\norm{F}_{\infty,[a,b]\times[c,d]},$
\end{enumerate}
for constants $C_{m_r,m_k}$ depending only on $m_r$ and $m_k$,
and constants $C_{n_r,n_\ell}$ depending only on $n_r$ and $n_\ell$.
\end{theorem}

\begin{proof}
Following the proof of Theorem \ref{thm3},
we start with the bound
$$|\mu_i^{m_k} \times \mu_j^{n_\ell} F|
 \leq C_{m_r} C_{n_r} \norm{F}_{\infty, [a,b] \times [c,d]}.$$
Therefore,
\begin{align*}
|\PP_{m_k}\QQ_{n_{\ell}} F(u,v)|
&=
\Big|
\sum_{i=0}^{m_k} \sum_{j=0}^{n_\ell} 
(\mu^{m_k}_{i} \times \mu^{n_\ell}_{j}) F 
\ D_i^{m_k}(u) D_j^{n_\ell}(v)
\Big| \\
&\leq 
 C_{m_r} C_{n_r} \norm{F}_{\infty, [a,b] \times [c,d]}
\sum_{i=0}^{m_k} |D_i^{m_k}(u)|
\sum_{j=0}^{n_\ell} |D_j^{n_\ell}(v)| \\
&\leq 
 C_{m_r} C_{n_r} \norm{F}_{\infty, [a,b] \times [c,d]}
(m_k+1)\norm{E_{m_k}^{m_r} (\alpha^k,:)^{-1}}_\infty
(n_\ell+1)\norm{E_{m_\ell}^{m_r} (\beta^\ell,:)^{-1}}_\infty \\
&= 
C_{m_r,m_k} C_{n_r,n_\ell} \norm{F}_{\infty, [a,b] \times [c,d]}
\end{align*}
with
$C_{m_r,m_k} := C_{m_r} (m_k+1)\norm{E_{m_k}^{m_r} (\alpha^k,:)^{-1}}_\infty$
and
$C_{n_r,n_\ell}  := C_{n_r} (n_\ell+1)\norm{E_{m_\ell}^{m_r} (\beta^\ell,:)^{-1}}_\infty$.
This establishes the first inequality. 
The second follows from
$$
\norm{\PP_{m_k} \oplus \QQ_{n_\ell} F}_{\infty, [a,b]\times[c,d]}  
 \leq
\norm{\PP_{m_k} F}_{\infty, [a,b]\times[c,d]}  
+\norm{\QQ_{n_\ell} F}_{\infty, [a,b]\times[c,d]}  
+\norm{\PP_{m_k} \oplus \QQ_{n_\ell} F}_{\infty, [a,b]\times[c,d]},
$$
for arbitrary $(x,y) \in [a,b] \times [c,d]$.
The third estimate follows the first estimate combined with:
\begin{align*}
\norm{\BB_{\m,\n}} &\leq \sum_{k=0}^r \norm{\PP_{m_k}\QQ_{n_{r-k}} }
  + \sum_{k=0}^{r-1} \norm{\PP_{m_k}\QQ_{n_{r-k-1}}}.
\end{align*}
\end{proof}

Now, we will construct a basis for $S_{\m,\n}$ that is dual to $\BB_{\m,\n}$.
The sequences $\alpha^k$ and $\beta^k$ can be viewed as bijective maps
$\alpha^k : [0, \ldots, m_k] \rightarrow \NN_0^{m_k+1}$
and
$\beta^k : [0, \ldots, n_k] \rightarrow \NN_0^{n_k+1}$,
with inverses $\halpha^k := (\alpha^k)^{-1}$ and $\hbeta^k := (\beta^k)^{-1}$
defined on $\ran(\alpha^k)$ and $\ran(\beta^k)$, respectively.
For example, if $\alpha^k = [2,4,8]$, 
then $\halpha^k(2) = 1$, 
$\halpha^k(4)= 2$
and
$\halpha^k(8)= 3$,
while $\halpha^k(3)$ is not defined.
Let
$$\PPhi^{\m,\n} := [\Phi_{i,j} : (i,j)\in G_{\m,\n}]$$
with
$$
\Phi_{i,j} := 
  \sum_{k=0}^r D_{\halpha^k(i)}^{m_k}\times D_{\hbeta^{r-k}(j)}^{n_{r-k}}
 -\sum_{k=0}^{r-1} D_{\halpha^k(i)}^{m_k}\times D_{\hbeta^{r-k-1}(j)}^{n_{r-k-1}},
$$
with $D_{\halpha^k(i)}^{m_k} := 0$ if $i \not\in\ran(\alpha^k)$
and $D_{\hbeta^k(j)}^{n_k} := 0$ if $j\not\in\ran(\beta^k)$.
Now, we let
$$\Lambda_{\m,\n} := [\lambda_{i}^{m_r} \times \lambda_{j}^{n_r} : (i,j) \in G_{\m,\n}].$$

\begin{theorem}
$\PPhi^{\m,\n}$ is a basis for $S_{\m,\n}$ dual to $\Lambda_{\m,\n}$.
\end{theorem}

\begin{proof}
Let $(p,q) \in G_{\m,\n}$ and $(i,j) \in G_{\m,\n}$.
Let $k_1 := \min\{k : i \in \ran(\alpha^k)\}$
and $k_2 := \min\{k : j \in \ran(\beta^k)\}$.
Then,
duality follows by:
\begin{align*}
(\lambda_{i}^{m_r} \times \lambda_{j}^{n_r}) \Phi_{p,q}
&= (\lambda_{i}^{m_r} \times \lambda_{j}^{n_r}) 
  \Big(\sum_{k=0}^r D_{\halpha^k(p)}^{m_k}\times D_{\hbeta^{r-k}(q)}^{n_{r-k}}
 -\sum_{k=0}^{r-1} D_{\halpha^k(p)}^{m_k}\times D_{\hbeta^{r-k-1}(q)}^{n_{r-k-1}}\Big)\\
&= \sum_{k=k_1}^{r-k_2}
  (\lambda_{i}^{m_r} D_{\halpha^k(p)}^{m_k})(\lambda_{j}^{n_r} D_{\hbeta^{r-k}(q)}^{n_{r-k}})
   -\sum_{k=k_1}^{r-k_2-1} 
  (\lambda_{i}^{m_r} D_{\halpha^k(p)}^{m_k})(\lambda_{j}^{n_r} D_{\hbeta^{r-k-1}(q)}^{n_{r-k-1}}) \\
&= \Big(\sum_{k=k_1}^{r-k_2} \delta_{i,p} \delta_{j,q}
   -\sum_{k=k_1}^{r-k_2-1} \delta_{i,p} \delta_{j,q} \Big)\\
&= \delta_{i,p} \delta_{j,q}.
\end{align*}
\end{proof}

Now that we have a dual basis, we can represent our discrete-blended 
quasi-interpolant as follows:
$$
\BB_{\m,\n} F := \Lambda_{\m,\n}^T \PPhi_{\m,\n} F
  = \sum_{(i,j) \in G_{\m,\n}} (\lambda_i^{m_r} \times \lambda_j^{n_r}) F \ \Phi_{i,j}.
$$

\begin{example}
To see how the construction works, consider $m = n = [2,4]$.
In Fig. \ref{f6a}, the quasi-uniform grids are plotted,
and in Fig. \ref{f6b} the discrete blended approximation to 
a function over these grids is plotted.
The construction involves the sum of two surfaces minus a third surface
(a so-called ``correction surface'').
For this example, the discrete blended approximation has the form.
\begin{align*}
&\BB_{[2,4],[2,4]} F (u,v) 
= \sum_{(i,j)\in G_{\m,\n}} b_{i,j} \Phi_{i,j}(u,v) \\
&\ \ = \sum_{i=0}^{4} \sum_{j=0}^{2} b_{i, 2j} D_{\halpha^1(i)}^4(u)D_{\hbeta^0(j)}^2(v)
+ \sum_{i=0}^{2} \sum_{j=0}^{4} b_{2i,j} D_{\halpha^0(i)}^2(u)D_{\hbeta^1(j)}^4(v)
- \sum_{i=0}^{2} \sum_{j=0}^{2} b_{2i,2j} D_{\halpha^0(i)}^2(u)D_{\hbeta^0(j)}^2(v).
\end{align*}
The dual basis arranged with respect to the grid points in $G_{\m,\n}$ can
be viewed as follows:
$$
\PPhi_{\m,\n} = 
\begin{array}{|l|l|l|l|l|}
\hline
    \BBB{0}{0}{4}{2} + \BBB{0}{0}{2}{4} - \BBB{0}{0}{2}{2}
       & \BBB{1}{0}{4}{2}
       & \BBB{2}{0}{4}{2} + \BBB{1}{0}{2}{4} - \BBB{2}{0}{2}{2}
       & \BBB{3}{0}{4}{2}
       & \BBB{4}{0}{4}{2} + \BBB{2}{0}{24} - \BBB{2}{0}{2}{2} \\
\hline
    \BBB{0}{1}{2}{4}  & &\BBB{1}{1}{2}{4} & &\BBB{2}{1}{2}{4} \\
\hline
    \BBB{0}{1}{4}{2} + \BBB{0}{2}{2}{4} - \BBB{0}{1}{2}{2}
      &\BBB{1}{1}{4}{2}
      &\BBB{2}{1}{4}{2} + \BBB{1}{2}{2}{4} - \BBB{2}{1}{2}{2}
      &\BBB{3}{1}{4}{2}
      &\BBB{4}{1}{4}{2} + \BBB{2}{2}{2}{4} - \BBB{2}{1}{2}{2} \\
\hline
    \BBB{0}{1}{3}{4} & &\BBB{1}{1}{3}{4} & &\BBB{2}{1}{3}{4} \\
\hline
    \BBB{0}{2}{4}{2} + \BBB{0}{4}{2}{4} - \BBB{0}{2}{2}{2}
      &\BBB{1}{2}{4}{2}
      &\BBB{2}{2}{4}{2} + \BBB{1}{4}{2}{4} - \BBB{2}{2}{2}{2}
      &\BBB{3}{2}{4}{2}
      &\BBB{4}{2}{4}{2} + \BBB{2}{4}{2}{4} - \BBB{2}{2}{2}{2}
\\
\hline
\end{array}
$$
\end{example}

\begin{figure}[H]
\centering
\scalebox{0.3}{
\begin{pspicture}(6,4)
\psdots[dotscale=2]
(0,0)(0,1)(0,2)(0,3)(0,4)
(1,0)     (1,2)     (1,4)
(2,0)(2,1)(2,2)(2,3)(2,4)
(3,0)     (3,2)     (3,4)
(4,0)(4,1)(4,2)(4,3)(4,4)
\end{pspicture}}
=
\ 
\scalebox{0.3}{
\begin{pspicture}(5,4)
\psdots[dotscale=2]
(0,0)(0,2)(0,4)
(1,0)(1,2)(1,4)
(2,0)(2,2)(2,4)
(3,0)(3,2)(3,4)
(4,0)(4,2)(4,4)
\end{pspicture}}
+
\ 
\scalebox{0.3}{
\begin{pspicture}(5,4)
\psdots[dotscale=2]
(0,0)(0,1)(0,2)(0,3)(0,4)
(2,0)(2,1)(2,2)(2,3)(2,4)
(4,0)(4,1)(4,2)(4,3)(4,4)
\end{pspicture}}
-
\ 
\scalebox{0.3}{
\begin{pspicture}(5,4)
\psdots[dotscale=2]
(0,0)(0,2)(0,4)
(2,0)(2,2)(2,4)
(4,0)(4,2)(4,4)
\end{pspicture}}
\caption{$\BB_{[2,4],[2,4]} = P_{2}Q_{4} + P_{4}Q_{2} - P_{2}Q_{2}$ onto 
$\Pi_{4}\otimes \Pi_2 + \Pi_{2}\otimes \Pi_4$.  }
\label{f6a}
\end{figure}

\begin{figure}[H]
\includegraphics[scale=1]{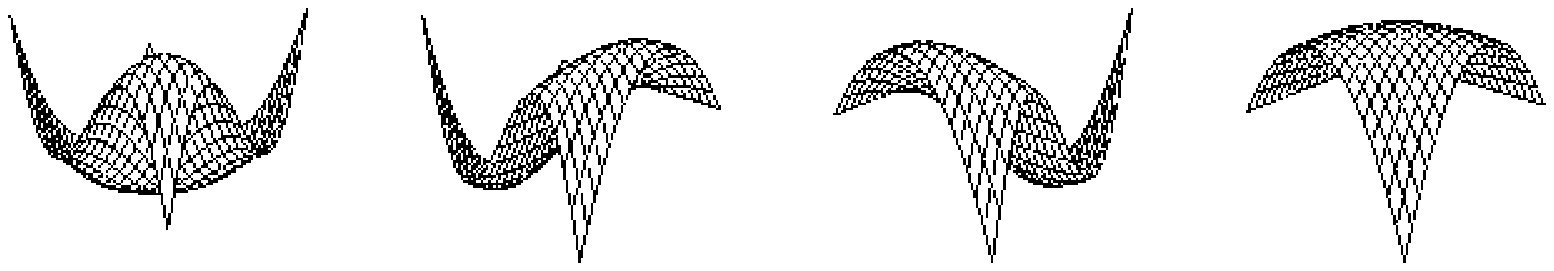}
\caption{$S = S_1+S_2-S_3$ for $\m=\n=[2,4]$.}
\label{f6b}
\end{figure}

\vskip 20pt
\begin{example}
In Figure \ref{f3}, we plot the dual basis functions $\Phi_{ij}$ and quasi-uniform grid
for the case $\m=\n=[1,2]$.
The approximating space is
$S_{\m,\n} = \Pi_1\otimes\Pi_2 + \Pi_2\otimes\Pi_1 - \Pi_1\otimes\Pi_1$.
The basis functions are plotted at the same position in the
grid corresponding to the indexing of the dual functionals.
\end{example}

\begin{figure}[H]
\begin{minipage}{.4 \textwidth}
\centering
\includegraphics[scale=0.4]{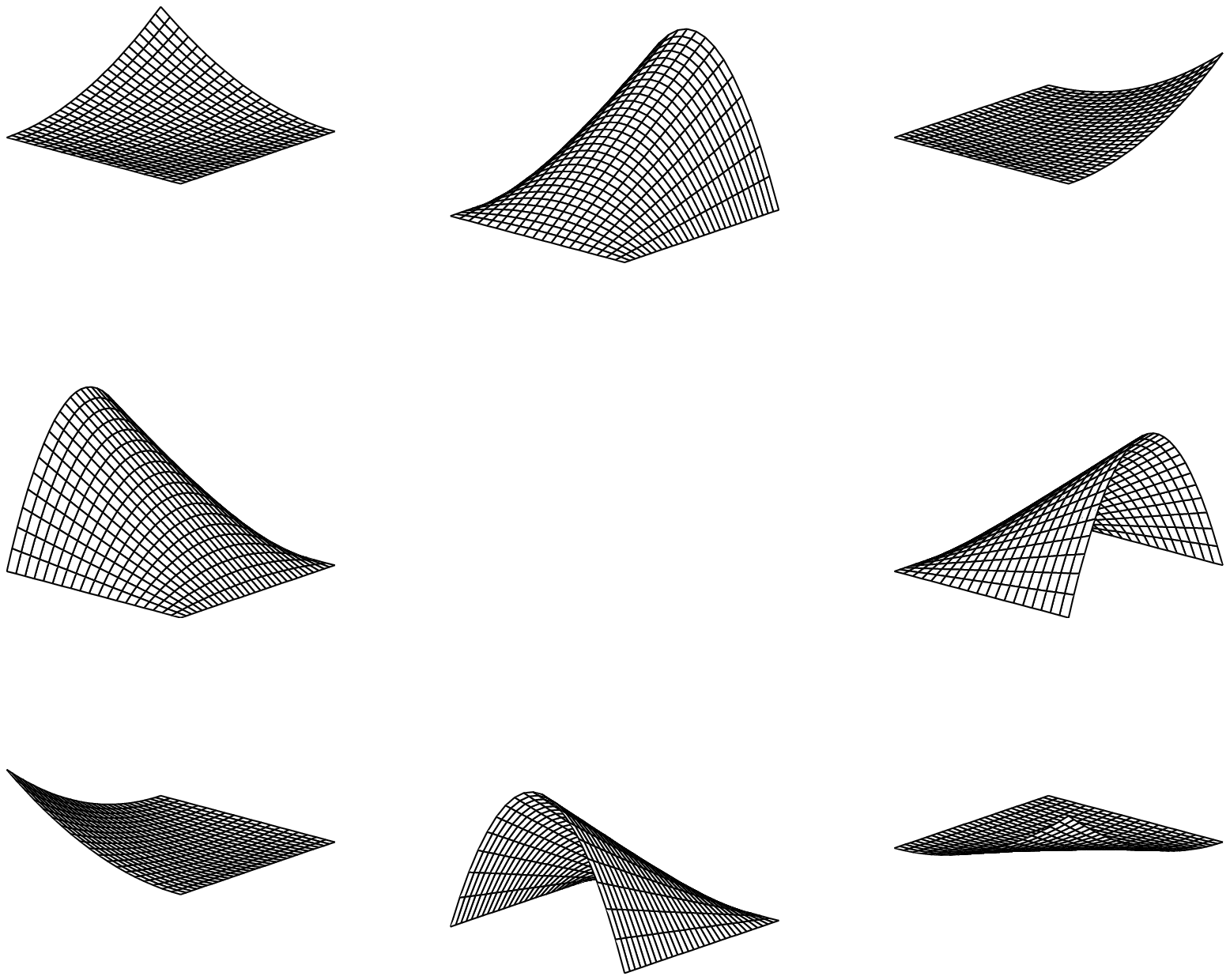}
\end{minipage}
\begin{minipage}{.4 \textwidth}
\centering
\scalebox{1}{
\begin{pspicture}(4,4)
\psdots[dotscale=2]
(0,0)(2,0)(4,0)
(0,2)     (4,2)
(0,4)(2,4)(4,4)
\end{pspicture}}
\end{minipage}
\caption{Basis functions and quasi-uniform grid for case $\m=\n=[1,2]$.}
\label{f3}
\end{figure}

\section{Approximation}

In this section we derive a rate of approximation for $\BB_{\m,\n}$
with $\m$ and $\n$ strictly increasing sequences in $\NN_0^{r+1}$
such that $m_k| m_{k+1}$ and $n_k | n_{k+1}$.
We begin with error estimates for the Boolean sum $\PP_{m_k} \oplus \QQ_{n_\ell}$.
for some $k$ and $\ell$.
Let
$$
d(F,\ran(\PP_{m_k} \oplus \QQ_{n_\ell}),[a,b]\times[c,d])
 := \min \{\norm{F-p}_{\infty,[a,b]\times[c,d]} : p \in \ran(\PP_{m_k} \oplus \QQ_{n_\ell})\}
$$
be the distance from a bivariate function $F$
to the range of the projector $\PP_{m_k} \oplus \QQ_{n_\ell}$
over the rectangle $[a,b]\times[c,d]$.

\begin{lemma}
Let $F$ be a continuous function on $[a,b] \times [c,d]$.
Then,
$$\norm{F-(\PP_{m_k} \oplus \QQ_{n_\ell})F}_{\infty,[a,b]\times[c,d]} 
\leq (1+\norm{\PP_{m_k}\oplus \QQ_{n_\ell}}) \
d(F,\ran(\PP_{m_k} \oplus \QQ_{n_\ell}),[a,b]\times[c,d])$$
\end{lemma}

\begin{proof}
Let $p \in \ran(\PP_{m_k} \oplus \QQ_{n_\ell})$.
Then,
\begin{align*}
\norm{F-(\PP_{m_k} \oplus \QQ_{n_\ell})F}_{\infty,[a,b]\times[c,d]} 
&\leq
\norm{F - p}_{\infty,[a,b]\times[c,d]} 
+ \underbrace{\norm{p - (\PP_{m_k} \oplus \QQ_{n_\ell})p}_{\infty,[a,b]\times[c,d]}}_0
+ \norm{(\PP_{m_k} \oplus \QQ_{n_\ell})(p-F)}_{\infty,[a,b]\times[c,d]}  \\
&\leq (1 + \norm{\PP_{m_k} \oplus \QQ_{n_\ell}}_{\infty,[a,b]\times[c,d]}) 
  \norm{F-p}_{\infty,[a,b]\times[c,d]}.
\end{align*}
Since true for all $p \in \Pi_{m_k} \times \Pi_{n_\ell}$, we have
$$
\norm{(F - \PP_{m_k}\oplus\QQ_{n_\ell}}_{\infty,[a,b]\times[c,d]} 
\leq
(1 + \norm{\PP_{m_k} \oplus \QQ_{n_\ell}}_{\infty,[a,b]\times[c,d]}) 
d(F,\ran(\PP_{m_k} \oplus \QQ_{n_\ell}),[a,b]\times[c,d]).
$$
\end{proof}

\begin{lemma}
Let $F\in C^{m_k+1,n_\ell+1}([a,b] \times [c,d])$ and suppose $b-a = d-c = h$.
Then,
$$
d(F,\ran(\PP_{m_k} \oplus \QQ_{n_\ell}),[a,b]\times[c,d]) \leq K_{m_k,n_\ell}(F) \ 
h^{m_k+n_\ell+2}
$$
with
$$
K_{m_k,n_\ell}(F) := \frac{\norm{F^{(m_k+1,n_\ell+1)}}_{\infty,[a,b] \times [c,d]}} 
  {2^{m_k+n_\ell+2}(m_k+1)!(n_\ell+1)!}.
$$
\end{lemma}

\begin{proof}
The univariate Taylor polynomials centered at 
$c_x := \frac{a+b}{2}$ and $c_y = \frac{c+d}{2}$ are, respectively,
\begin{align*}
T_{m_k}(x) &= \sum_{i=0}^{m_k} f^{(i)}(c_x) \frac{(x-c_x)^i}{i!} \\
T_{n_\ell}(y) &= \sum_{j=0}^{n_\ell} g^{(j)}(c_y) \frac{(y-c_y)^j}{j!},
\end{align*}
with remainder estimates
\begin{align*}
f(x) &= T_{m_k}(x) + f^{({m_k}+1)}(\xi) \frac{(x-c_x)^{{m_k}+1}}{({m_k}+1)!} \\
g(y) &= T_{n_\ell}(y) + g^{({n_\ell}+1)}(\eta)\frac{(y-c_y)^{{n_\ell}+1}}{({n_\ell}+1)!}.
\end{align*}
for some $\xi$ depending on $x$ that lies between $x$ and $c_x$,
and for some $\eta$ depending on $y$ that lies between $y$ and $c_y$.
Let $\TT_{m_k} = T_{m_k} \times I$ and $\TT_{n_\ell} = I \times T_{n_\ell}$ be bivariate extensions.
Then,
$$
\TT_{m_k} \oplus \TT_{n_\ell} = \TT_{m_k} + \TT_{n_\ell} - \TT_{n_\ell}\TT_{m_k}.
$$
Let $F \in C^{{m_k}+1,{n_\ell}+1}$.
Then
\begin{align*}
\TT_{m_k} \oplus \TT_{n_\ell} &= \sum_{i=0}^{m_k} F^{(i,0)}(c_x,c_y) \frac{(x-c_x)^i}{i!}
 +\sum_{j=0}^{n_\ell} F^{(0,j)}(c_x,c_y) \frac{(y-c_y)^j}{j!} 
 -\sum_{i=0}^{m_k} \sum_{j=0}^{n_\ell} F^{(i,j)}(c_x,c_y) \frac{(x-c_x)^i}{i!} \frac{(y-c_y)^j}{j!}.
\end{align*}
Note that 
$$\ran(\TT_{m_k} \oplus \TT_{n_\ell}) = 
 \Pi_{m_k} \times I + I \times \Pi_{n_\ell} + \Pi_{m_k} \times \Pi_{n_\ell}
= \ran(\PP_{m_k} \oplus \QQ_{n_\ell}).
$$
By the error in Boolean sum interpolation,
$$
F - \TT_{m_k} \oplus \TT_{n_\ell} F = \TT_{m_k}^c \TT_{n_\ell}^c F = (I - \TT_{m_k})(I-\TT^{n_\ell})F,
$$
with
$$
\TT_{m_k}^c \TT_{n_\ell}^c F 
   = F^{({m_k}+1,{n_\ell}+1)}(\xi,\eta) \frac{(x-c_x)^{{m_k}+1}}{({m_k}+1)!} \frac{(y-c_y)^{{n_\ell}+1}}{({n_\ell}+1)!}
$$
for some $\xi$ between $x$ and $c_x$,
and some $\eta$ between $y$ and $c_y$.
Hence,
\begin{align*}
d(F,\ran(\PP_{m_k} \oplus \QQ_{n_\ell}),[a,b]\times[c,d]) &\leq 
\norm{F(x,y)-\TT_{m_k} \oplus \TT_{n_\ell} F(x,y)}_{\infty,[a,b]\times[c,d]} \\
  &\leq \norm{F^{({m_k}+1,{n_\ell}+1)}}_{\infty,[a,b]\times[c,d]}
   \frac{h^{{m_k}+1}}{2^{{m_k}+1}({m_k}+1)} 
   \frac{h^{n_\ell+1}}{2^{{n_\ell}+1}({n_\ell}+1)!}.
\end{align*}
Note that the powers of $2$ in the denominator come from choosing
$c_x$ and $c_y$ at the centers of $[a,b]$ and $[c,d]$, respectively.
\end{proof}

As a corollary to the previous two lemmas, we have the following theorem:

\begin{theorem}
\label{thm1}
Let $F\in C^{{m_k}+1,{n_\ell}+1}([a,b] \times [c,d])$
with $b-a=d-c = h$.
Then,
$$\norm{F-(\PP_{m_k} \oplus \QQ_{n_\ell})F}_{\infty,[a,b]\times[c,d]} 
\leq K_{{m_k},{n_\ell}}(F) h^{{m_k}+{n_\ell}+2} 
$$
with
$$
K_{m,n}(F) := (1+\norm{\PP_{m}\oplus \QQ_{n}}) 
\frac{\norm{F^{(m+1,n+1)}}_{\infty,[a,b] \times [c,d]}} {2^{m+n+2}(m+1)!(n+1)!}.
$$
\end{theorem}

\begin{prop}
(by Proposition 2 of section 1.4 from \cite{DS89})
Let $\PP_{m_k}^{c}F = F - \PP_{m_k}F$, $\QQ_{n_\ell}^{c}F = F - \QQ_{n_\ell} F$
and $\PP_{m_k}^c\QQ_{n_\ell}^c F = F - \PP_{m_k} \oplus \QQ_{n_\ell} F$.
Then,
\begin{equation}
\BB_{\m,\n}^c = \sum_{k=0}^{r+1} \PP_{m_{k-1}}^c\QQ_{n_{r-k}}^c - 
   \sum_{k=0}^{r} \PP_{m_k}^c\QQ_{n_{r-k}}^c,
\end{equation}
with $\PP_{m_{-1}} := \II$ and $\QQ_{n_{-1}} := \II$.
\end{prop}

From this and the previous theorem, we arrive at our main result of this section.

\begin{theorem}
\label{thm2}
Let $F\in C^{m_r+1,n_r+1}([a,b] \times [c,d])$ with $b-a=d-c = h$.
Then,
$$
\norm{F - \BB_{\m,\n}F}_{\infty,[a,b]\times[c,d]}  \leq C_{\m,\n}(F)\ h^{p}
$$
with 
$$
p = \min\{m_{k-1} + n_{r-k} +2 : \ k=0, \ldots, r+1\}, 
$$
$m_{-1} := -1 =: n_{-1}$,
with respect to the constant
$$
C_{\m,\n}(F) = \sum_{k=0}^{r+1} K_{m_{k-1},n_{r-k}}(F)+
  \sum_{k=0}^{r} K_{m_{k},n_{r-k}}(F),
$$
with
$$
K_{m_k,n_\ell}(F) := 
(1+\norm{\PP_{m_k}\oplus \QQ_{n_\ell}})  \
\frac{\norm{F^{(m_k+1,n_\ell+1)}}_{\infty,[a,b] \times [c,d]}} 
  {2^{m_k+n_\ell+2}(m_k+1)!(n_\ell+1)!},
$$
and 
$$
\norm{\PP_{m_k}\oplus \QQ_{n_\ell}}  \leq
 C_{m_r,m_k} + C_{n_r,n_\ell} + C_{m_r,m_k} C_{n_r,n_\ell},
$$
for constants $C_{m_r,m_\ell}$ and $C_{n_r,n_\ell}$ 
defined in the proofs of Theorems \ref{thm3} and \ref{thm4}
depending only on $m_r$, $m_\ell$, $n_r$ and $n_\ell$.
\end{theorem}

\begin{proof}
By Theorem \ref{thm1},
\begin{align}
\norm{F - \BB_{\m,\n} F}_{\infty,[a,b] \times [c,d]} &\leq 
\sum_{k=0}^{r+1} \norm{\PP_{m_{k-1}}^c\QQ_{n_{r-k}}^cF}_{\infty,[a,b] \times [c,d]} 
  +\sum_{k=0}^{r} \norm{\PP_{m_k}^c\QQ_{n_{r-k}}^c F}_{\infty,[a,b] \times [c,d]} \\
&\leq \sum_{k=0}^{r+1} K_{m_{k-1},n_{r-k}}(F) h^{m_{k-1}+n_{r-k}+2} 
  + \sum_{k=0}^{r} K_{m_{k},n_{r-k}}(F) h^{m_{k}+n_{r-k}+2}.
\end{align}
Since $\m$ is increasing, 
$$ 
m_{k-1}+n_{r-k}  < m_{k}+n_{r-k},
$$
and so the lowest power that occurs is
$$
p = \min\{m_{k-1}+n_{r-k}+2\}.
$$
Since we require, when $k=0$, that $m_{-1} + n_r+2 = n_r+1$,
then $m_{-1} := -1$.
Likewise, $n_{-1} := -1$.
Therefore,
\begin{align}
\norm{F - \BB_{\m,\n} F}_{\infty,[a,b] \times [c,d]} &\leq 
   \Big(\sum_{k=0}^{r+1} K_{m_{k-1},n_{r-k}}(F) 
  + \sum_{k=0}^{r} K_{m_{k},n_{r-k}}(F) \Big) \ h^{p}.
\end{align}
The bounds on $\norm{\PP_{m_k}\oplus \QQ_{n_\ell}}$ come from Theorem \ref{thm4}.
\end{proof}

In the case $r=0$, $\m = [m]$ and $\n=[n]$ each contain just one number.
Then, we have the following special case of Theorem \ref{thm2}.

\begin{cor}
Assume that $\m = [m]$ and $\n=[n]$ for $r=0$.
Then $\BB_{\m,\n}F$ is a tensor product approximant to $F$ of approximation order
$p =  \min\{n+1, m+1\}$.
\label{cor1}
\end{cor}

\begin{proof}
In this case, $\BB_{\m,\n}$ reduces to the tensor product $\PP_{m}\QQ_{n}$
of $\PP_{m}$ and $\QQ_{n}$.
By Theorem \ref{thm2}, $p =  \min\{-1+n+2, m+(-1)+2\} = \min\{n+1, m+1\}$.
\end{proof}

\begin{example}
Suppose that $\m = [3,6,12]$ and $\n = [2,4,8]$.
Then, we rewrite as $[-1,3,6,12]$ and $[-1,2,4,8]$,
and the approximation order is
$$
p = 2 + \min\{-1+8, 3+4, 6+2, 12-1\}+2 =9.
$$
The full tensor product approximation with $\m=[12]$ and $\n=[8]$
has the same rate of approximation, i.e., 
$p =  \min\{n+1, m+1\} =  \min\{13,9\} = 9$.
Hence, we achieve the same rate of approximation for
$\BB_{[3,6,12],[2,4,8]}$ as by $\BB_{[12],[8]}$,
but with much fewer grid points.
This situation is discussed further in the next section.
\end{example}

\section{Serendipity elements}

The Serendipity elements are a class of finite elements in finite element
analysis that achieve a rate of approximation  better than one would expect.
These correspond to configurations that achieve the same order of approximation
as tensor product approximants, but with mainly boundary data 
(i.e., fewer interior points in the grid).
For our quasi-interpolants, we have an analogous situation.

Note that $\BB_{\m,\n}$ reduces to tensor product approximation when
$\m$ and $\n$ consist of just one number, i.e., when $r=0$.
By Corollary \ref{cor1}, the rate of approximation in tensor product approximation is
$p = \min\{n+1, m+1\}$.
In this case, $r=0$ and $\m = [m]$ and $\n = [n]$ for some $m$ and $n$.
The Serendipity elements are those that
achieve this rate of approximation on quasi-uniform grids.
In Tbl. \ref{f4a}, the rates of approximation are given for the
Serendipity elements corresponding to $\m = \n = [1,2]$, $[1,3]$, $[1,2,4]$, $[2,4]$.
They are $O(h^p)$ with $p = 3$, $4$, $5$ and $5$.
In each case, $S_{\m,\n}$ embeds in the tensor product space $\PP_{m_r} \PP_{n_r}$,
which  produces the same rates of approximation.
In Fig. \ref{f4b}, the quasi-uniform grids $G_{\m,\n}$ are plotted for these
cases.

\begin{table}[H]
\begin{center}
\begin{tabular}{|l|l|l|l|l|}
\hline
$\m$  & $\n$ & $S_{\m,\n}$ & $\dim(S_{\m,\n})$ & Approximation Order $p$,  $O(h^p)$ \\
\hline
$[1,2]$  & $[1,2]$ 
& $\Pi_1 \otimes \Pi_2 + \Pi_2 \otimes \Pi_1$
& $8$ & $2 + \min\{-1+2, 1+1, 2-1\} = 3$ \\
\hline
$[1,3]$ & $[1,3]$ 
& $\Pi_1 \otimes \Pi_3 + \Pi_3 \otimes \Pi_1$
& $12$ & $2 + \min\{-1+3, 1+1, 3-1\} = 4$ \\
\hline
$[1,2,4]$ & $[1,2,4]$ 
& $\Pi_1 \otimes \Pi_4 + \Pi_2 \otimes \Pi_2 + \Pi_4 \otimes \Pi_1$
& $17$ & $2 + \min\{-1+4, 1+2, 2+1 4-1\} = 5$ \\
\hline
$[2,4]$ & $[2,4]$ 
& $\Pi_2 \otimes \Pi_4 + \Pi_4 \otimes \Pi_2$
& $21$ & $2 + \min\{-1+4, 2+2, 4-1\} = 5$ \\
\hline
\end{tabular}
\end{center}
\caption{Approximation order for $\m=\n = [1,2]$, $[1,3]$, $[1,2,4]$, $[2,4]$.}
\label{f4a}
\end{table}

\begin{figure}[H]
\scalebox{.7}{
\begin{pspicture}(4,3)
\psdots[dotscale=1]
(0,0)(1,0)(2,0)
(0,1)     (2,1)
(0,2)(1,2)(2,2)
\end{pspicture}}
\scalebox{.7}{
\begin{pspicture}(5,4)
\psdots[dotscale=1]
(0,0)(1,0)(2,0)(3,0)
(0,1)(3,1)
(0,2)(3,2)
(0,3)(1,3)(2,3)(3,3)
\end{pspicture}}
\scalebox{.7}{
\begin{pspicture}(6,5)
\psdots[dotscale=1]
(0,0)(1,0)(2,0)(3,0)(4,0)
(0,1) (4,1)
(0,2)(2,2)(4,2)
(0,3)(4,3)
(0,4)(1,4)(2,4)(3,4)(4,4)
\end{pspicture}}
\scalebox{.7}{
\begin{pspicture}(5,5)
\psdots[dotscale=1]
(0,0)(1,0)(2,0)(3,0)(4,0)
(0,1)(2,1)(4,1)
(0,2)(1,2)(2,2)(3,2)(4,2)
(0,3)(2,3)(4,3)
(0,4)(1,4)(2,4)(3,4)(4,4)
\end{pspicture}}
\caption{Serendipity Elements for $\m=\n= [1,2]$, $[1,3]$, $[1,2,4]$, $[2,4]$.}
\label{f4b}
\end{figure}
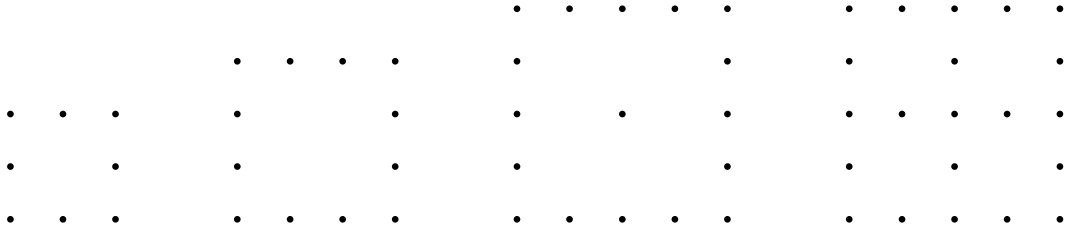


\section{Additional Examples}

In the following example we apply our construction to a well-known test function.
In Fig. \ref{f5a}, the tensor product Bernstein polynomial 
that approximates the function and the uniform grid is plotted for the case $\m=\n = 20$.
In Fig. \ref{f5b}, the discrete blended polynomial approximant for $\m = \n = [5,20]$
is plotted along with the corresponding quasi-uniform grid.
In the plots, the dots are control points.
In both cases the boundary curves are B\'ezier curves.
More importantly, with much less data, the discretely blended surface
does a very good job of approximating the full tensor product,
even though the approximation orders are different in this case 
($21$ for the tensor product compared to $12$ for the discretely blended surface).

\begin{figure}[H]
\includegraphics[scale=0.4]{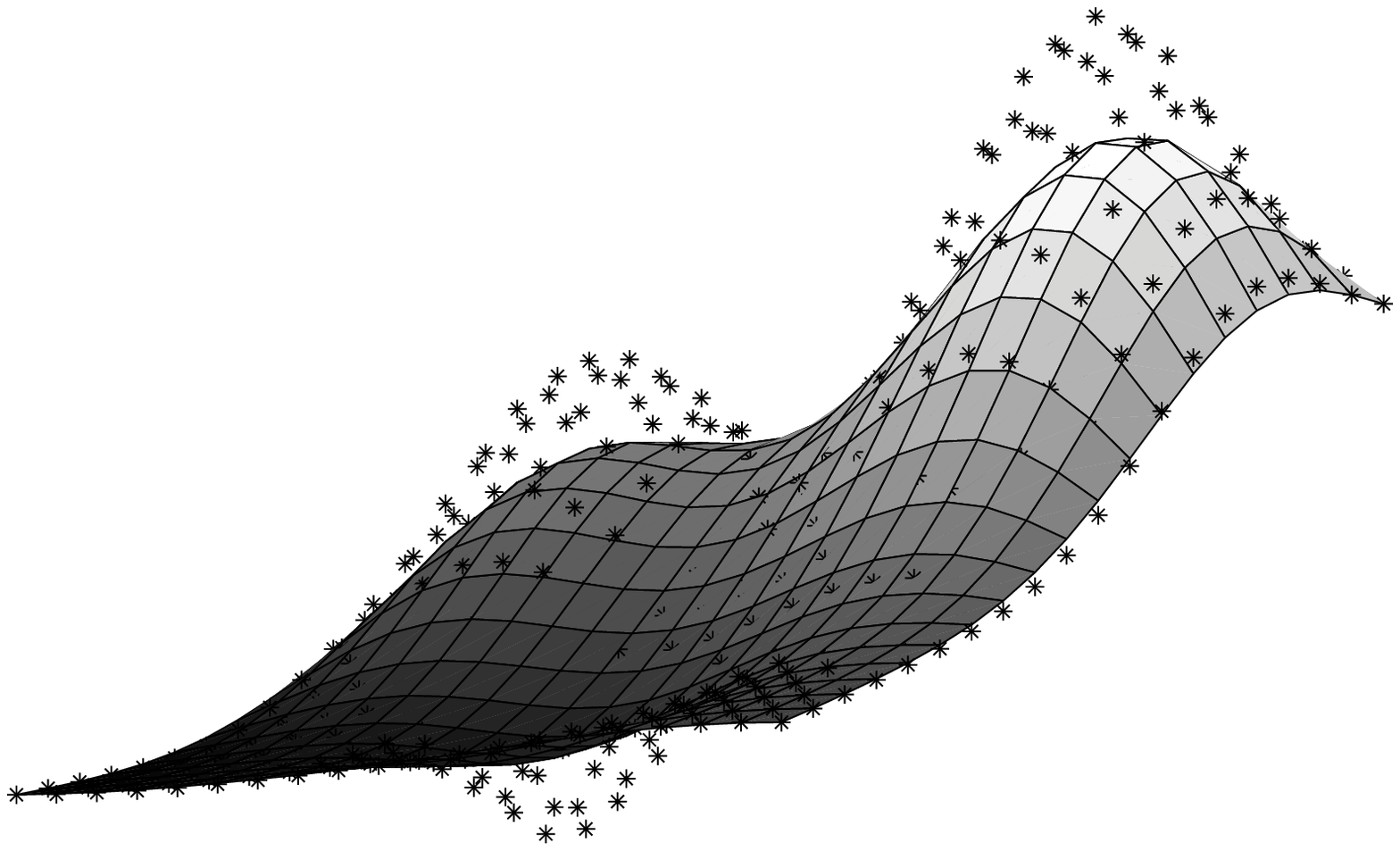}
\qquad
\includegraphics[scale=0.4]{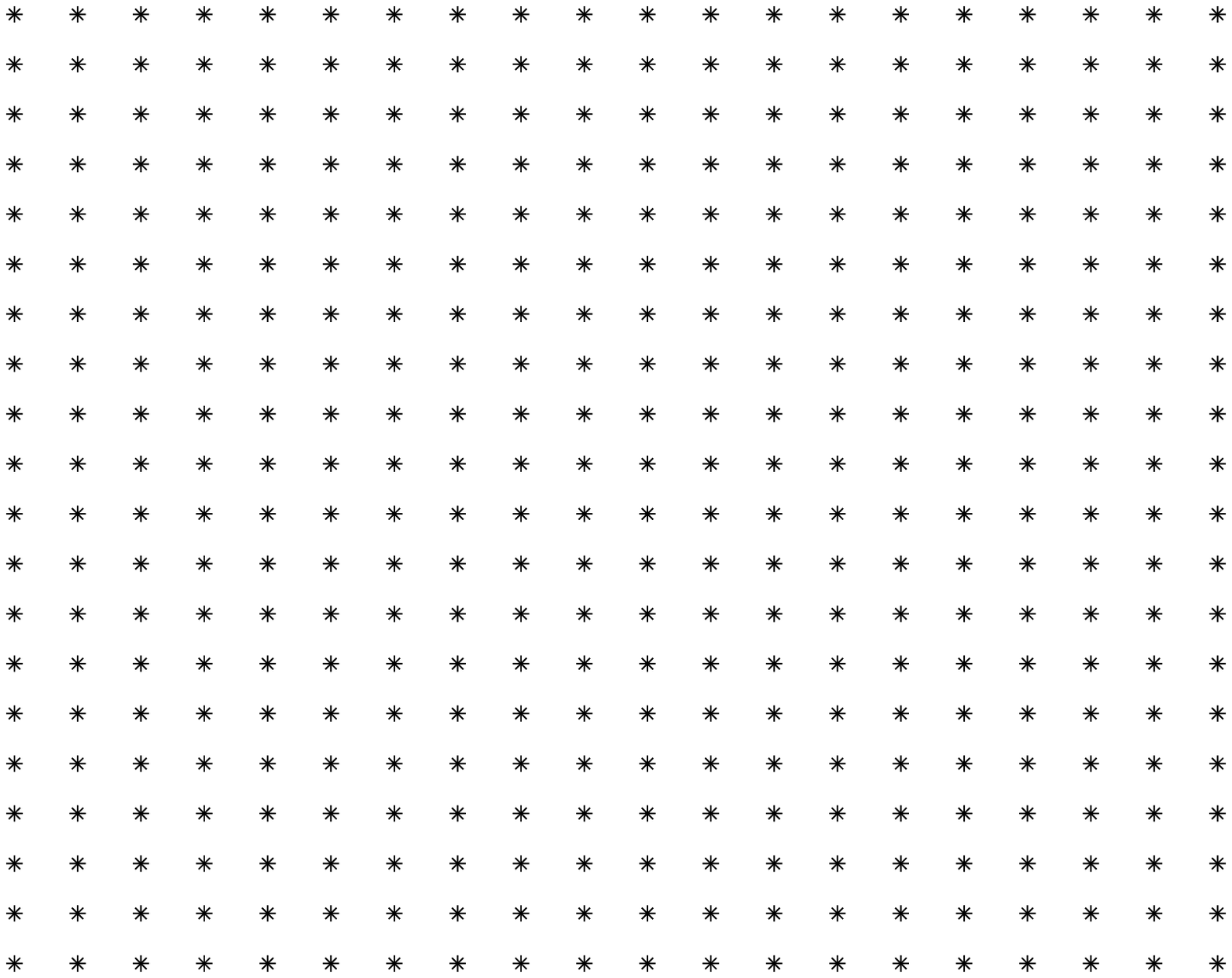}
\caption{Tensor Product Bernstein/B\'ezier Surface: $\m=\n=[20]$.}
\label{f5a}
\end{figure}

\begin{figure}[H]
\includegraphics[scale=0.4]{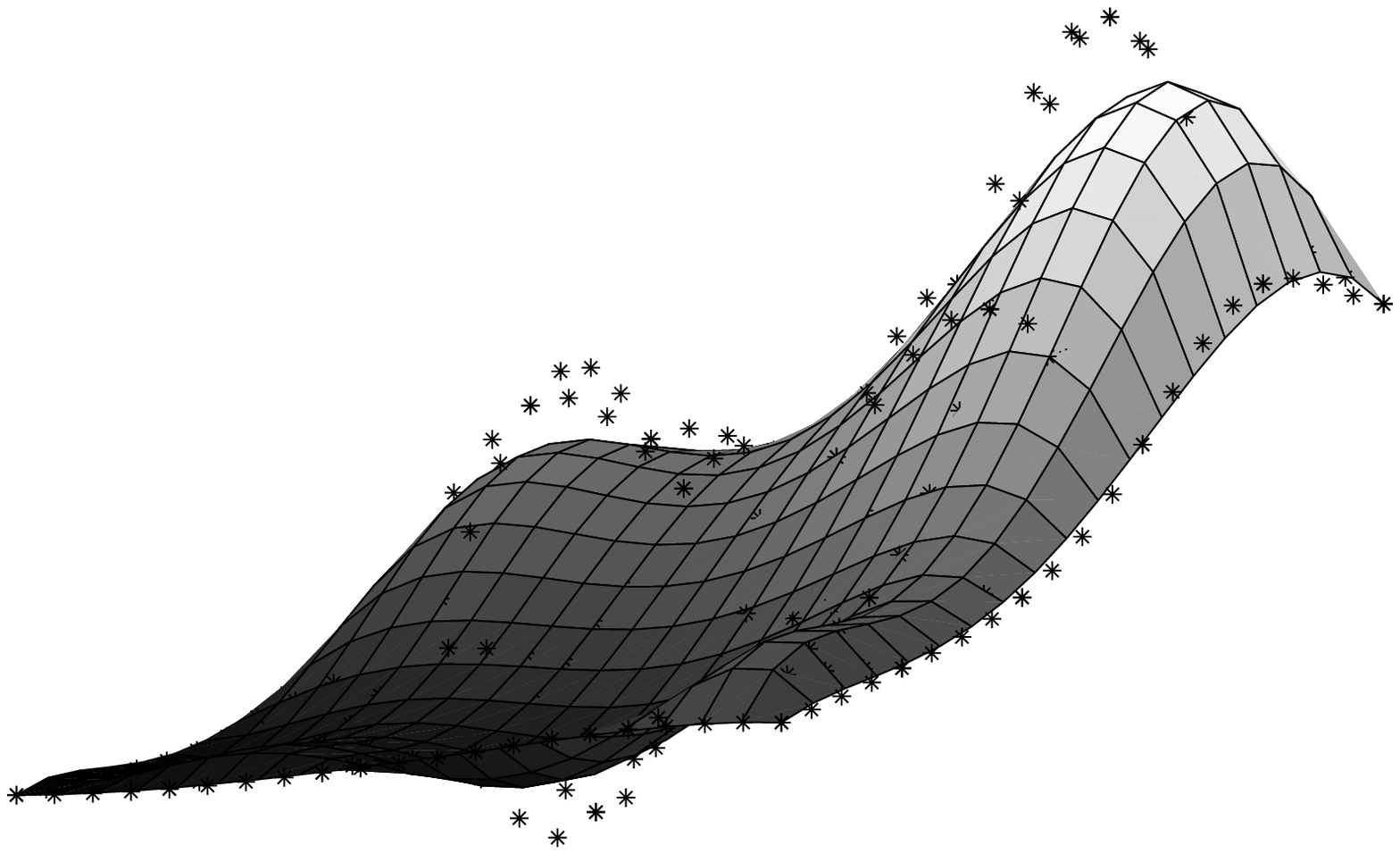}
\qquad
\includegraphics[scale=0.4]{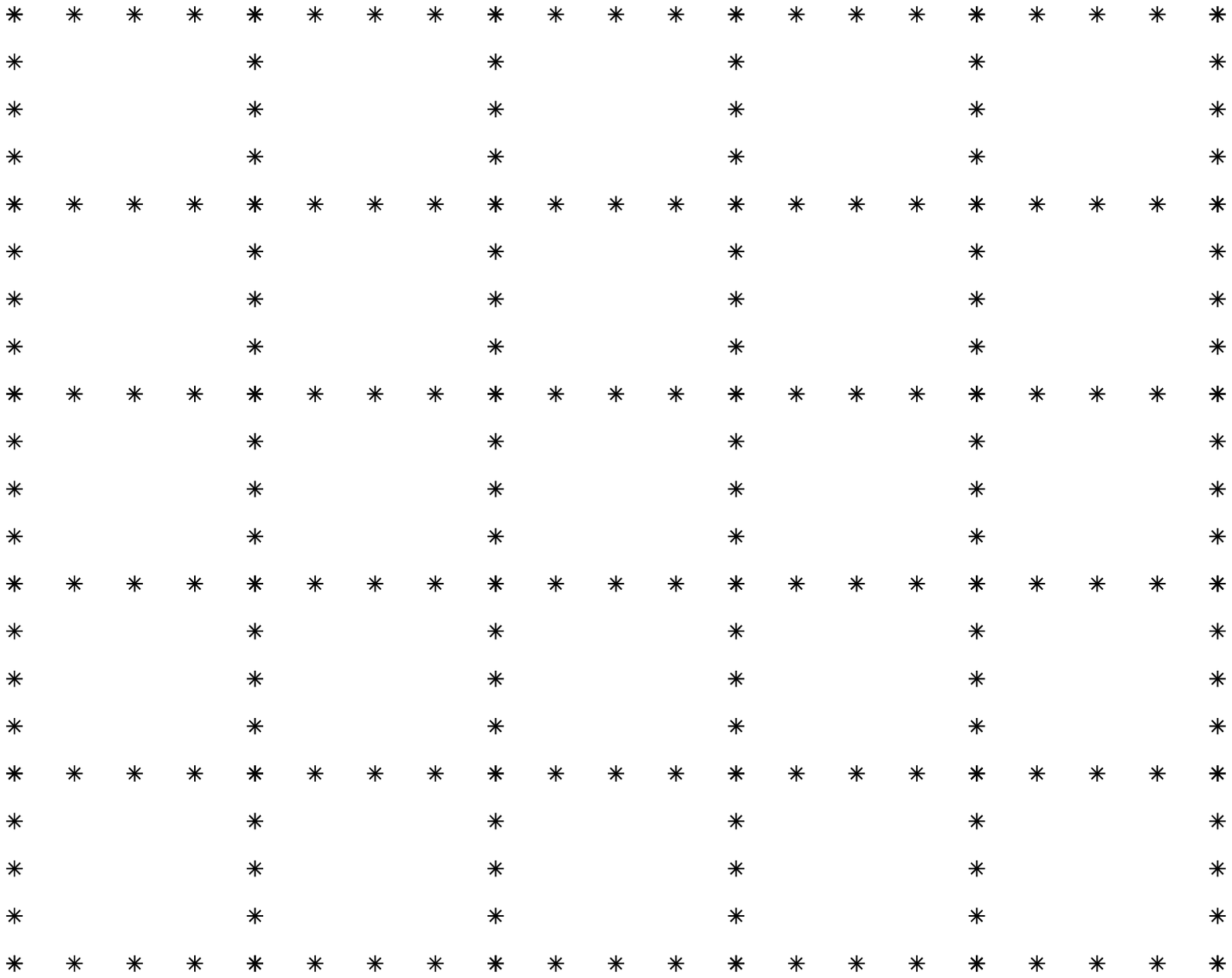}
\caption{Discretely Blended Surface: $\m=\n=[5,20]$.}
\label{f5b}
\end{figure}

In the next example we verify the rates of approximation for a 
piecewise discrete blended polynomial approximation of the function $f(x,y) = \sin(2xy)$.
This is plotted in Figure \ref{f6a} for the serendipity configuration $\m=\n=[2,4]$
for both \ref{f6a}, $4 \times 4$ and $16 \times 16$ piecewise polynomial grids.
Hence, for the second grid, the spacing $h$ is one quarter of the first, 
and we expect a much better rate of approximation.

\begin{figure}[H]
\centering
\includegraphics[width=3in]{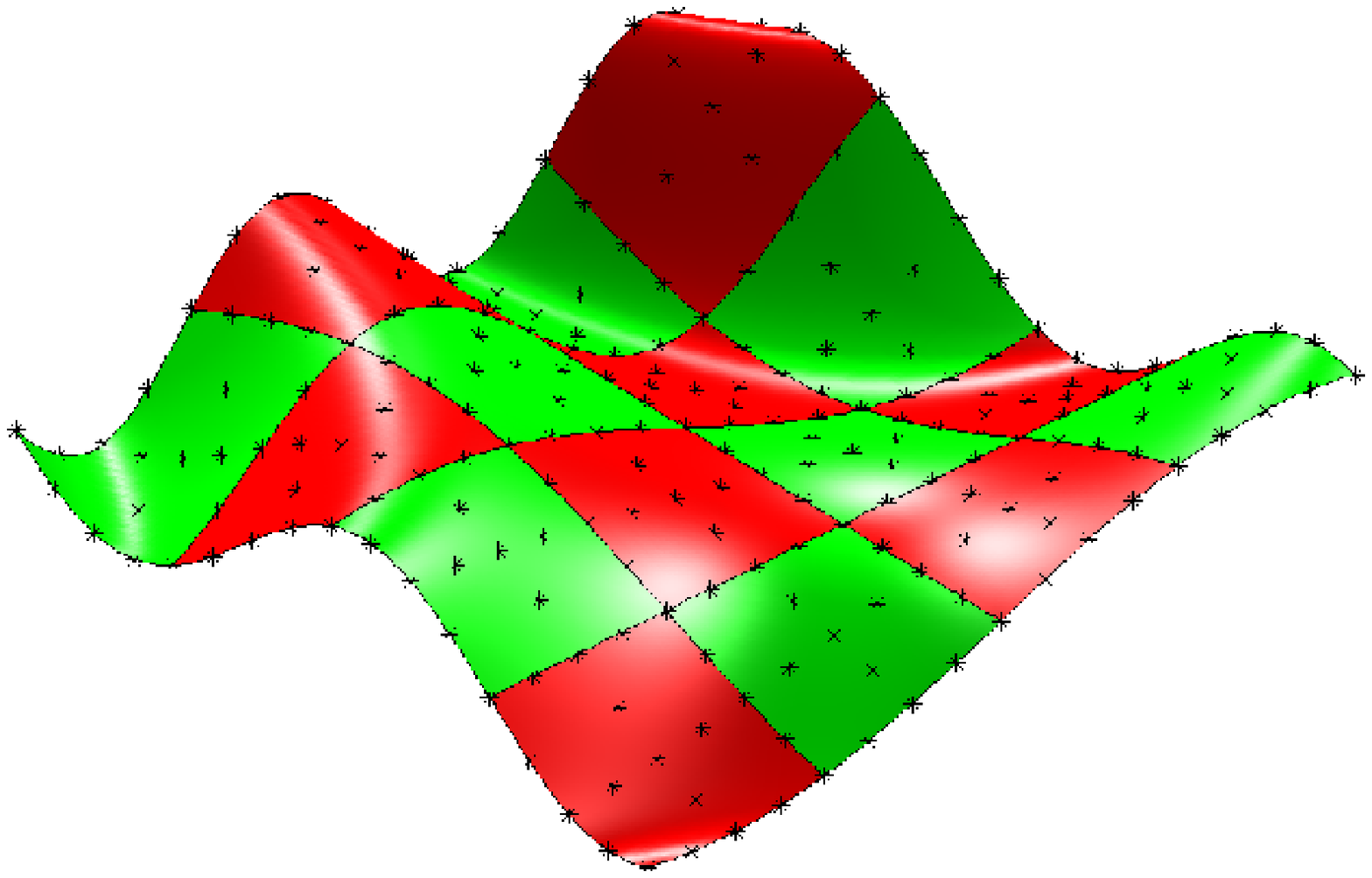}
\qquad
\includegraphics[width=3in]{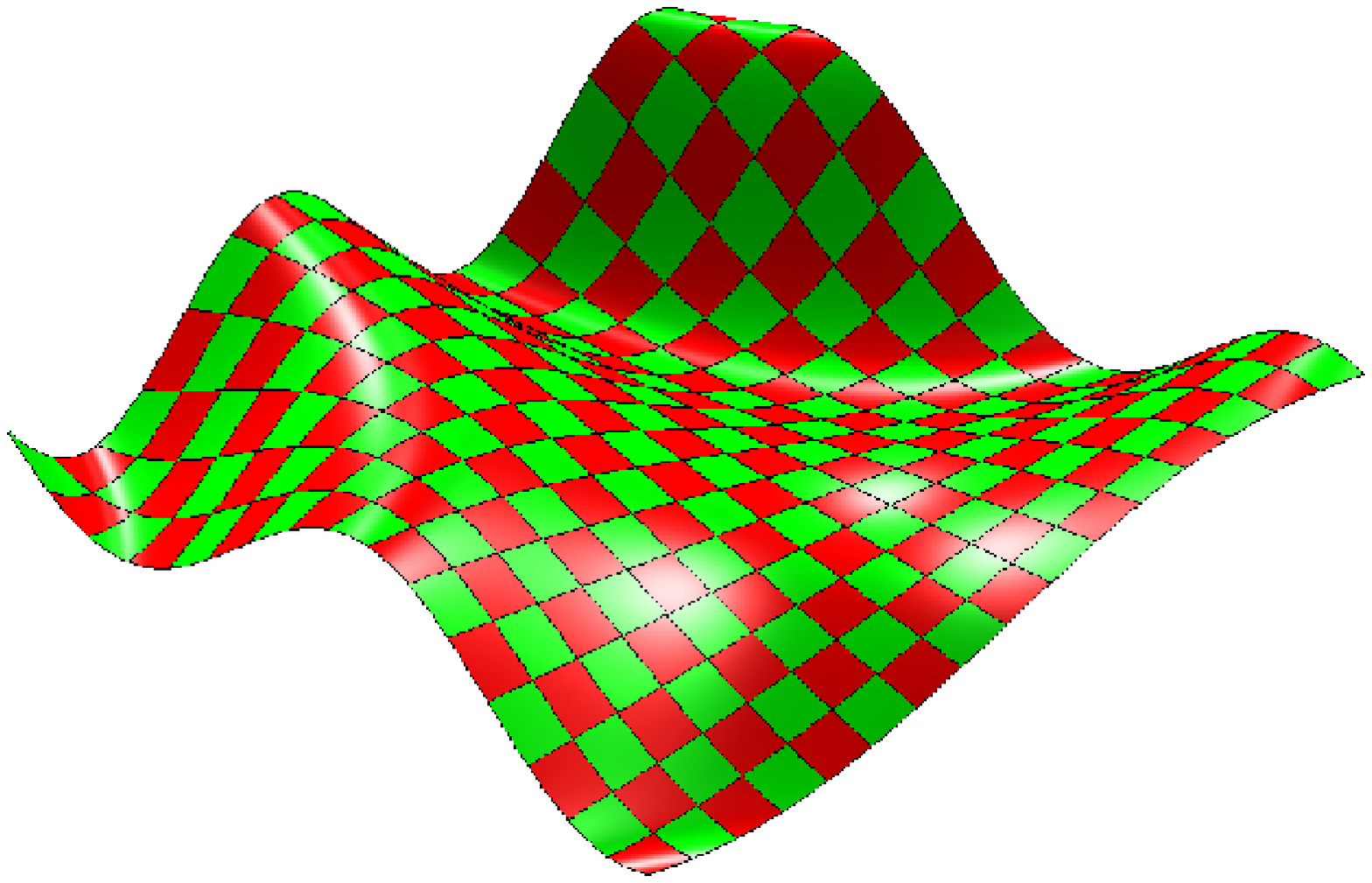}
\caption{$f(x,y) = \sin(2xy)$ with $\m=\n=[2,4]$: $4 \times 4$ and $16 \times 16$ Grids.}
\label{6a}
\end{figure}

The errors in approximation in this example for $k \times k$ grid approximations
for $k=1, \ldots, 16$ is tabulated in Table \ref{f6b} for the serendipity
configuration $\m=\n=[2,4]$ and for the tensor product $\m=\n=[4]$.
By the theory, these should achieve the same rate of approximation $h^5$.
The data confirms that we are close to this number.
The exact rates are calculuted as follows:
\begin{itemize}
\item $e_k \approx K h^p$ with $h = \frac{b-a}{k} = \frac{d-c}{k}$
\item Theoretical for Quasi-interpolant with $\m=\n=[2,4]$ on each rectangle:
$p = 2 + \min\{ -1+4, 2+2, 4-1\} = 5$.
\item Theoretical for Quasi-interpolant with $\m=\n=[4]$ on each rectangle:
$p = 2 + \min\{ -1+4, 4-1\} = 5$.
\end{itemize}

{\scriptsize
\begin{table}
\begin{center}
\begin{tabular}{|c||c|c|c|c|}
\hline
Method & \multicolumn{3}{c}{Error for piecewise approximation $1\times 1$ to $16 \times 16$} & \\
\hline
Quasi-Interpolant & 2.3370 & 0.5431 & 0.0453 & 0.0105 \\
m=n=[2,4] & 0.0027 & 0.0010 & 4.9070e-04 & 2.4617e-04 \\
$h^{4.6008}$ & 1.3123e-04 &  7.3995e-05 & 4.4315e-05 & 2.7561e-05 \\
& 1.9278e-05 & 1.3300e-05 & 9.4613e-06 & 6.7413e-06 \\
\hline
Quasi-Interpolant  &0.9977 &0.1564 &0.0248 &0.0030  \\
m=n=[4]  &0.0022 &8.9636e-04 &4.1602e-04 &1.8871e-04  \\
$h^{4.2938}$ &1.1996e-04 &7.1195e-05 &4.4315e-05 &2.7561e-05 \\
&1.9278e-05 &1.3300e-05 &9.4613e-06 &6.7413e-06 \\
\hline
\end{tabular}
\end{center}
\label{6b}
\end{table}
}

\section{Closing Remarks}

In this paper we have constructed a new quasi-interpolant for discrete blended
surface approximation based on dual basis functions in the Bernstein basis,
and we have established error estimates comparable
to approximation on full tensor product grids,
much like the Serendipity elements in the finite element literature.

Throughout this paper we assumed 
$m_{k} | m_{k+1}$ and $n_{k} | n_{k+1}$.
If we do not have this, we can still apply our construction by first using degree
elevation.
For example, if $\m = [2,3,5,7]$
we would degree elevation to get $\tilde\m = [2,4,8,16]$,
and then proceed as before.

The framework for this originates from the talk
``Dual bases on subspaces and the approximation from sums of polynomial and spline spaces'',
given by the author at the conference ``Mathematical Methods for Curves and Surfaces'', 
held in Oslo Norway, June, 2012.
The talk included constructions for Hermite and spline blended elements
similar to the construction in this paper.
Our plan is to publish the details of these results in forthcoming papers.

\bibliographystyle{amsplain}

\newpage 

\renewcommand\thefigure{\arabic{figure}}
\renewcommand\thetable{\arabic{table}}

\end{document}